\documentclass[11pt]{article}

\usepackage[utf8]{inputenc}
\pdfoutput=1
\usepackage[margin=1.1in,letterpaper]{geometry}
\usepackage[round]{natbib}
\usepackage{tikz}

\makeatletter
\newlength\aftertitskip     \newlength\beforetitskip
\newlength\interauthorskip  \newlength\aftermaketitskip

\def\maketitle{\par
 \begingroup
   \def\thefootnote{\fnsymbol{footnote}}
   \def\@makefnmark{\hbox to 4pt{$^{\@thefnmark}$\hss}}
   \@maketitle \@thanks
 \endgroup
\setcounter{footnote}{0}
 \let\maketitle\relax \let\@maketitle\relax
 \gdef\@thanks{}\gdef\@author{}\gdef\@title{}\let\thanks\relax}

\def\@startauthor{\noindent \normalsize\bf}
\def\@endauthor{}
\def\@starteditor{\noindent \small {\bf Editor:~}}
\def\@endeditor{\normalsize}
\def\@maketitle{\vbox{\hsize\textwidth
 \linewidth\hsize \vskip \beforetitskip
 {\begin{center} \LARGE\@title \par \end{center}} \vskip \aftertitskip
 {\def\and{\unskip\enspace{\rm and}\enspace}%
  \def\addr{\small\it}%
  \def\email{\hfill\small\tt}%
  \def\name{\normalsize\bf}%
  \def\AND{\@endauthor\rm\hss \vskip \interauthorskip \@startauthor}
  \@startauthor \@author \@endauthor}
}}

\makeatother

\pdfoutput=1                    

\usepackage{amsmath,amsthm}
\usepackage{graphicx}
\usepackage{color}
\usepackage{amssymb,amsfonts,amsxtra,mathrsfs,bm}
\usepackage{multicol}
\usepackage{multirow}
\usepackage{paralist}
\usepackage{xspace}
\usepackage{algorithm}%
\usepackage{algpseudocode}%
\usepackage[colorlinks=true, urlcolor = blue, citecolor=blue]{hyperref}
\usepackage{bbm}
\usepackage{nicefrac}

\newcommand{\reals}{\mathbb{R}}

\def\defas{\stackrel{\text{def}}{=}}
\newcommand{\tr}{\text{tr}}
\newcommand{\diag}{\textbf{diag}}

\newcommand{\SG}{\mathcal{SG}(\mathbb{P}_d)}

\DeclareMathOperator*{\argmin}{argmin}

\newtheorem{theorem}{Theorem}[section]
\newtheorem{lem}[theorem]{Lemma}
\newtheorem{prop}[theorem]{Proposition}
\newtheorem{cor}[theorem]{Corollary}

\theoremstyle{definition}
\newtheorem{defn}[theorem]{Definition}

\newtheorem{rmk}[theorem]{Remark}
\newtheorem{example}[theorem]{Example}

\usepackage{booktabs} 
\usepackage{makecell} 

\numberwithin{equation}{section}

\newcommand{\nat}{\mathbb{N}}

\newcommand{\pd}{\mathbb{P}_d}

\newcommand{\Xc}{\mathcal{X}}
\newcommand{\Mc}{\mathcal{M}}

\newcommand{\R}{\mathcal{R}}

\title{Structured Regularization for Constrained Optimization \\ on the SPD Manifold}
\author{\name Andrew Cheng \email{andrewcheng@g.harvard.edu}\\
\name Melanie Weber \email{mweber@seas.harvard.edu}\\
  \addr{Harvard University}
}

\begin{document}
\maketitle

\begin{abstract}
    Matrix-valued optimization tasks, including those involving symmetric positive definite (SPD) matrices, arise in a wide range of applications in machine learning, data science and statistics. Classically, such problems are solved via constrained Euclidean optimization, where the domain is viewed as a Euclidean space and the structure of the matrices (e.g., positive definiteness) enters as constraints. More recently, geometric approaches that leverage parametrizations of the problem as unconstrained tasks on the corresponding matrix manifold have been proposed. While they exhibit algorithmic benefits in many settings, they cannot directly handle additional constraints, such as inequality or sparsity constraints. A remedy comes in the form of constrained Riemannian optimization methods, notably, Riemannian Frank-Wolfe and Projected Gradient Descent. However, both algorithms require potentially expensive subroutines that can introduce computational bottlenecks in practise.  To mitigate these shortcomings, we introduce a class of structured regularizers, based on symmetric gauge functions, which allow for solving  constrained optimization on the SPD manifold with faster unconstrained methods. We show that our structured regularizers can be chosen to preserve or induce desirable structure, in particular convexity and ``difference of convex'' structure. We demonstrate the effectiveness of our approach in numerical experiments.
\end{abstract}

\section{Introduction}\label{sec1}
We study constrained optimization problems of the form
\begin{equation}\label{eq:prob}
	\min_{x \in \Xc \subset \pd} \phi(x) \; ,
\end{equation}
where $\phi: \pd \rightarrow \R$ is a smooth function defined on the 
symmetric, positive definite matrices $\pd$ and $\Xc \subset \pd$ a subset defined by geometric constraints. 
Problems of this form arise in many settings, including the computation of Tyler's M-estimators~\citep{tyler1987distribution,m-scatter,wiesel2012geodesic}, robust subspace recovery~\citep{zhang2016robust}, barycenters problems~\citep{sra2013sdivergence}, matrix-square roots~\citep{sra2015matrixsquareroot}, the computation of Brascamp-Lieb constants~\citep{weber2022computing}, and learning determinantal point processes~\citep{mariet2015fixed}, among others.

Classical approaches for this class of problems include constrained Euclidean optimization, where the domain in problem~\ref{eq:prob} is viewed as a Euclidean space and the geometric structure of the problem enters as constraints.  However, it is often beneficial to encode the positive definiteness constraint explicitly in the parametrization of the domain by solving problem~\ref{eq:prob} as a constrained problem on the manifold of symmetric positive definite matrices (SPD manifold). For instance, if the objective $\phi$ is geodesically convex with respect to the Riemannian metric, then one can provide global optimality certificates for first-order methods in the Riemannian setting. Consequently, several constrained Riemannian optimization methods have been proposed, including variants of Riemannian Projected Gradient Descent (R-PGD)~\citep{liu2019simple} and projection-free Riemannian Frank-Wolfe (R-FW) methods~\citep{frank-wolfe,weber2021projection}. 
However, several shortcomings arise, which limit the applicability of those methods in practise.  First, both R-PGD and R-FW rely on subroutines for implicitly imposing constraints, which can be costly in the geometric setting.  Second,  the geometric tools needed to implement Riemannian optimization methods, including Riemannian gradients, exponential maps, and parallel transport operators, often introduce significant computational overhead compared to their Euclidean counterparts.  To mitigate both limitations, we will apply a \emph{mixed Euclidean-Riemannian perspective}, which leverages the computational efficiency of Euclidean methods while admitting global optimality certificates thanks to geodesic convexity.

We propose a class of regularizers based on symmetric gauge functions, which allows for relaxing several types of constraints that frequently arise in geometric optimization.  We show that this \emph{structured regularization} of constrained optimization tasks can preserve desirable properties, such as geodesic convexity and difference of convex (DC) structure.  Moreover, in some settings,  DC structure may be induced by a suitable regularization. Optimization tasks with DC objectives can be solved using \emph{Convex-Concave Procedures} (short: \emph{CCCP}),  a class of Euclidean solvers that can often numerically outperform classical first-order methods in practise~\citep{sra2013sdivergence,wiesel2012geodesic}.  In settings where the DC objectives is geodesically convex, we can leverage a Riemannian analysis to obtain global optimality certificates~\citep{pmlr-v202-weber23a}.  We will show that this lens applies readily to our regularized objectives, allowing us to leverage CCCP with global optimality guarantees in the constrained setting. To the best of our knowledge, this represents the first application of CCCP to constrained, geodesically convex programs. Importantly, our structured regularizers are highly modular, which simplifies the design of new regularizers for a variety of programs. 

We present a convergence analysis of CCCP applied to our regularized objectives and discuss the computational complexity of CCCP, specifically with regards to the complexity of a crucial subroutine, the \emph{CCCP oracle}, which lies at the heart of the algorithm. We propose effective solvers for this subroutine for several classical SPD optimization tasks.  
Our results allow for recovering and reinterpreting previously known specialized fixed point approaches for related problems in our structured regularization framework. 
We corroborate our theoretical finding with numerical experiments, which illustrate the competitive performance of the proposed methods.

\subsection{Related Works}

\paragraph{Riemannian optimization} Optimization on geometric domains has received significant attention in the machine learning and statistics communities motivated by a wide range of applications that involve structured data and models. This has led to the generalizations of many classical Euclidean algorithms to the Riemannian setting, including for geodesically convex~\citep{udriste1994convex,Bacak+2014,zhang2016first} and nonconvex~\citep{boumal2019global} problems, as well as constrained~\citep{frank-wolfe,liu2019simple} and stochastic~\citep{bonnabel2013stochastic,zhang2016first,Riemannian_SVRG,weber2021projection} settings. Constrained Riemannian optimization has largely focused on projection-based~\citep{liu2019simple} and projection-free~\citep{frank-wolfe} first-order methods, which ensure the feasibility of the iterates via subroutines that can be expensive in practise. To the best of our knowledge, handling constraints via regularizers that are explicitly designed to preserve desirable structure in the objective (geodesic convexity, DC structure) has not been studied systematically in the prior Riemannian optimization literature.

\paragraph{Regularization in SPD optimization} 
Regularization techniques have been studied for many problems on the manifold of positive definite matrices. Notable examples include covariance estimation~\citep{bien,structuredcovestimation_wiesel}, where the regularizer enforces sparsity constraints and 
optimistic likelihood estimation, where side information is leveraged for regularization~\citep{Nguyen2019CalculatingOL}. Usually, the regularizer is designed with a specific task in mind. In contrast, here we aim to design a general class of regularizers and investigate its properties for preserving desirable structure in the objective. 

\paragraph{DC optimization}
The optimization of DC function has been extensively studied in the Euclidean optimization literature. The CCCP algorithm has emerged as a popular solver for problems of this structure. More recently, geometric optimization problems with DC structure have received interest, including differences of g-convex functions~\citep{cruzneto,souza2015proximal,ferreira2021difference} and differences of Euclidean convex functions that are g-convex~\citep{pmlr-v202-weber23a}. While the former rely on Riemannian tools, such as exponential maps and Riemannian gradients, the latter can be implemented using purely Euclidean tools~\citep{pmlr-v202-weber23a}. To the best of our knowledge no extensions of CCCP to constrained geometric problems have been considered in the prior literature.

\subsection{Summary of Contributions}
We briefly summarize the main contributions of our work.
\begin{enumerate}
    \item We introduce a class of structured regularizers for constrained optimization on the manifold of symmetric positive definite matrices. Our regularizers are based on symmetric gauge functions, whose inherent algebraic properties provide the regularizers with a modular structure. This allows for the design of custom regularizers for a range of constrained problems.
    \item We show that structured regularizers can be chosen to preserve or induce desirable structure in the objective, in particular difference of convex structure and geodesic convexity. The former allows for leveraging a simple CCCP approach to find solutions, and the latter to guarantee their global optimality.
    \item To the best of our knowledge, we provide the first analysis of Euclidean CCCP applied to constrained geodesically convex programs. This approach has notable  computational benefits compared to existing constrained Riemannian optimization approaches.
    \item We illustrate the utility of our approach in several sets of numerical experiments, highlighting the computational efficiency and numerical stability of the proposed CCCP methods in applications.
\end{enumerate}

\section{Background}

\subsection{Riemannian Geometry of $\pd$}\label{sec:background}
Throughout this paper we consider the set of real symmetric square matrices with strictly positive eigenvalues, denoted by 
\[
\pd \defas \{ X \in \reals^{d\times d}: X \succ 0 \}.
\]
 A \textit{manifold} $\mathcal{M}$ is a topological space that is locally Euclidean with a tangent space $\mathcal{T}_x \mathcal{M}$ associated to each point $x \in \mathcal{X}$. If $\mathcal{M}$ is \textit{smooth} and has a smoothly varying inner product $\langle u, v \rangle_x$ defined on $\mathcal{T}_x \mathcal{M}$ for $x \in \mathcal{M}$ then it is a \textit{Riemannian manifold}. 

 Here, we focus on the algorithmic benefits of viewing $\pd$ under both the \textit{affine-invariant Riemannian} geometry and the \textit{Euclidean} geometry. When endowed with the \textit{affine-invariant} inner product 
 \[\langle A, B\rangle_X=\operatorname{tr}\left(X^{-1} A X^{-1} B\right) \quad X \in \mathbb{P}_d, A, B \in T_X\left(\mathbb{P}_d\right)=\mathbb{H}_d \; ,\]
the positive definite matrices form a Riemannian manifold. 
Here, the tangent space $\mathbb{H}_d$ is the space of $d\times d$ real symmetric matrices. Under this geometry, given two points $A, B \in \pd$ there is an explicit parametrization for the \textit{unique geodesic} that interpolates $A$ to $B$ given by 
 \begin{equation}\label{eq:intro_gcvx_def}
    \gamma(t)=A^{1 / 2}\left(A^{-1 / 2} B A^{-1 / 2}\right)^t A^{1 / 2}, \quad 0 \leq t \leq 1.
\end{equation}
The geodesic given in \eqref{eq:intro_gcvx_def} is referred to as the \textit{weighted geometric mean} of $A$ and $B$. The midpoint of the geodesic denoted by $A \sharp B \defas \gamma(1/2)$ is known as the \textit{geometric mean} of $A$ and $B$. Furthermore, the \textit{Riemannian metric} corresponding to this geometry is given by 
\begin{equation}\label{eq:riem-metric}
    \delta_R(A, B)=\left\|\log A^{-1 / 2} B A^{-1 / 2}\right\|_F.
\end{equation}
It is important to remark that the resulting Riemannian manifold $\pd$ is a \textit{Cartan-Hadamard manifold}, i.e., it is \textit{complete}, \textit{simply connected}, and has \textit{non-positive curvature}.  The Cartan-Hadamard manifold setting is particularly suitable for geodesically convex analysis and optimization \citep{Bacak+2014}.

The Euclidean geometry of $\pd$ is induced by endowing the symmetric positive definite matrices with the smooth inner product 
\[
\langle A, B \rangle = \tr(A^\top B) \qquad \forall A,B \in \pd \; ,
\]
in which case the corresponding \textit{Euclidean metric} is the Frobenius norm 
\[
d(A,B) = \|A - B\|_F.
\]
In this case, we can view the set $\pd$ as a \textit{convex cone}, i.e., a set closed under conic combinations. This conic perspective lends itself to convex analysis and optimization \citep{Nesterov1994InteriorpointPA}. 

We further provide definitions for several convexity and smoothness notions that will be used throughout the paper.\\

\begin{defn}[Geodesic convexity of sets]
We say that a set $S \subseteq \pd$ is \textit{geodesically convex} (short: g-convex) if for any two points $A,B \in \pd$, the unique geodesic $\gamma:[0,1] \to \pd$ given by \eqref{eq:intro_gcvx_def} lies entirely in $S$, i.e., the image satisfies $\gamma([0,1]) \subseteq S$.    \\
\end{defn}

\begin{defn}[Geodesic convexity of functions]
    We say that $\phi: S \to \reals$ is a \textit{geodesically convex function} if $S \subseteq \pd$ is geodesically convex and $f \circ \gamma :[0,1] \to \reals$ is (Euclidean) convex for each geodesic segment $\gamma :[0,1] \to \pd$ whose image is in $S$ with $\gamma(0) \neq \gamma(1)$.
\end{defn}

 Recall that the Riemannian gradient on $\pd$ can be computed as follows. Let $\bar{\phi}: \mathbb{H}_d \to \reals$ be a function over symmetric matrices with the Euclidean metric from $\reals^{n \times n}$. Denote $\phi$ as the restriction $\phi \defas \bar{\phi}\vert_{\pd}$ to the SPD manifold equipped with the invariant metric. Then the Riemannian gradient of $\phi$ and the Euclidean gradient of $\bar{\phi}$ are related as
$\operatorname{grad}\phi(X) = X \operatorname{grad}\bar{\phi}(X)X$ for all $(X,V) \in \mathcal{T}_X\left(\pd\right)$.\\

\begin{defn}
    If $\phi: \pd \to \reals$ is differentiable then its gradient $\operatorname{grad} \phi(X)$ is defined as the unique vector $V \in \mathcal{T}_X \left(\pd\right)$ with $D\phi(X) [V] = \langle \operatorname{grad}\phi(X), V \rangle_X$.\\
\end{defn}

Although the Euclidean perspective provides fast algorithms with provable convergence guarantees, many classical optimization tasks are nonconvex under the Euclidean metric. However, a large subset of these tasks admit a geodesically convex formulation with respect to the Riemannian metric allowing provable convergence to global optima. Unfortunately, Riemannian optimization techniques can often introduce nontrivial computational overhead stemming from the cost of computing geometric tools, such as geodesics and Riemannian gradients. By adopting a \textit{mixed Euclidean-Riemannian perspective} one can simultaneously reap the computational benefits of the Euclidean perspective and the theoretical guarantees obtained via a Riemannian analysis.

\subsection{Difference of Convex (DC) Optimization}\label{sec:background-cccp}
Optimization tasks on the positive definite matrices frequently exhibit a special structure, where the objective function can be written as a difference of two convex functions. Formally, we consider instances of problem~\ref{eq:prob}, where $\phi(x) = f(x) - h(x)$ with $f(\cdot), h(\cdot)$ Euclidean convex and $h(\cdot)$ smooth. The convexity of $h(\cdot)$ directly implies that
\begin{equation*}
  -h(x) \le -h(y) - \langle \nabla h(y), x-y \rangle \; .
\end{equation*}
We can use this inequality to upper bound the objective, which defines the following surrogate function:
\begin{equation*}
  \phi(x) \le Q(x, y) \defas f(x)-h(y) - \langle \nabla h(y),x-y \rangle \; .
\end{equation*}
The idea of convex-concave procedures (short: CCCP) is to iteratively minimize this surrogate function instead of the original, non-convex objective (Algorithm~\ref{alg:cccp}). Notably, this algorithm is purely Euclidean and does not require the computation of Riemannian tools, such as exponential maps or parallel transport operators.

With a purely Euclidean analysis one can show that this algorithm converges asymptotically to a stationary point of the underlying objective~\citep{lanckriet2009convergence}, but due to non-convexity as non-asymptotic convergence analysis is challenging in the general case. However, if $\phi(\cdot)$ is in addition geodesically convex, then sublinear, global convergence guarantees can be obtained for the (purely Euclidean) CCCP algorithm:

\begin{theorem}[\citep{pmlr-v202-weber23a}]
\label{prop:cccp-conv}
Let $d(x_0,x^*) \leq R$ for some $x_0 \in \Mc$ with $\phi(x) \leq \phi(x_0)$. If the functions $Q(x,x_k)$ in Alg.~\ref{alg:cccp} are first-order surrogate functions, then
\begin{equation}
    \phi(x_k) - \phi(x^*) \leq \frac{2L\alpha_{\Mc}^2(R)}{k+2} \qquad \forall k \geq 1 \; ,
\end{equation}
where $\alpha_{\Mc}$ depends on the geometry of the manifold and $L$ characterizes the smoothness of $h(\cdot)$.
\end{theorem}
\begin{algorithm}
\small 
 \caption{Convex-Concave Procedure (CCCP)} 
  \label{alg:cccp}
  \begin{algorithmic}[1]
    \State \textbf{Input:} $x_0 \in \Mc$, K
    \For{$k=0,1,\ldots,K-1$}
        \State Let $Q(x, x_k) = f(x) - h(x_k) - \langle \nabla h(x_k),x-x_k \rangle$.
        \State $x_{k+1} \gets \argmin_{x \in \Mc} Q(x,x_k)$.
    \EndFor
    \State \textbf{Output:} $x_K$
  \end{algorithmic}
\end{algorithm}

We note that the CCCP algorithm cannot explicitly handle constraints in its standard form. The structured regularization techniques introduced in this work will allow for applying CCCP approaches to constrained tasks.

\subsection{Symmetric Gauge Functions and Unitarily Invariant Norms on $\pd$}
In this section we recall classical results on symmetric gauge functions and their algebraic structure. The properties of this class of functions will form the basis for the design of our structured regularizers, introduced in the next section. A more comprehensive overview of symmetric gauge functions can be found in~\citep{LIM2013115} and~\citep[Ch. IV]{bhatia97}.

We begin with a formal definition of \emph{symmetric gauge functions}:

        \begin{defn}[Symmetric Gauge Functions.] A function $\Phi: \reals^d \to \reals_+$ is called a \emph{symmetric gauge function} if 
        \begin{enumerate}
            \item $\Phi$ is a norm.
        \item $\Phi(\sigma_d(x)) = \Phi(x)$ for all $x \in \reals^d$ and all permutation maps $\sigma_n: \reals^d \to \reals^d$. This is known as the \emph{symmetric property}.
        \item $\Phi(\alpha_1 x_1, \ldots, \alpha_d x_d) = \Phi(x_1, \ldots, x_d)$ for all $x \in \reals^d$ and $\alpha_k \in \{\pm 1\}$. This is known as the \emph{gauge invariant} or \emph{absolute property}.
        \end{enumerate}
    \end{defn}

It is easy to see that symmetric gauge functions are closed under positive scaling, which directly follows from the fact that it is a norm.

    \begin{prop}[Closure under positive scaling]
        If $\Phi:\reals^d \to \reals$ is a symmetric gauge function then $\alpha \Phi(x)$ is a symmetric gauge function for all $\alpha > 0$.
    \end{prop}

Observe that the class of symmetric gauge functions is closed under addition.

    \begin{prop}[Closure under addition]
        If $\Phi_1, \ldots \Phi_n$ are symmetric gauge functions, then so is $\Phi = \Phi_1 + \cdots + \Phi_n.$ 
    \end{prop}

\begin{proof}
    We prove this for the case $n=2$. We know that if $\Phi_1$ and $\Phi_2$ are norms on $\reals^n$ and then so is $\Phi = \Phi_1 + \Phi_2$. Let $P: \reals^n \to \reals^n$ be a permutation matrix and fix some $x \in \reals^n$. Then 
    \[
    \Phi(Px) \defas \Phi_1(Px) + \Phi_2(Px) = \Phi_1(x) + \Phi_2(x) \defas \Phi(x). 
    \]
    This establishes the symmetric property of $\Phi$.
    To show gauge invariance, let $\{\alpha_j\}_{j=1}^n \subseteq \{\pm 1\}$. Then
    \begin{align*}
        \begin{split}
            \Phi(\alpha_1 x_1, \ldots, \alpha_n x_n)  &=  \Phi_1(\alpha_1 x_1, \ldots, \alpha_n x_n) +  \Phi_2(\alpha_1 x_1, \ldots, \alpha_n x_n)
            \\&= \Phi_1(x_1, \ldots, x_n) + \Phi_2(x_1, \ldots, x_n) 
            \\& = \Phi(x).
        \end{split}
    \end{align*}
    Thus $\Phi$ satisfies all the properties of a symmetric gauge invariant function. We can induct to prove the general case $n > 2$.
\end{proof}

Recall the following definition of the dual of a norm on $\reals^d$.

    \begin{defn}
        If $\Phi$ is a norm on $\mathbb{R}^d$ the dual of $\Phi$ is denoted by $\Phi^{*}: \reals^d \to \reals_+$ and defined as 
        \[
        \Phi^*(x) \defas \sup_{\{y \in \mathbb{R}^d : \Phi(y) \leq 1\}} | \langle x,y \rangle |.
        \]
        where $\langle \cdot, \cdot \rangle$ is the cannonical inner product of $\reals^d$.
    \end{defn}

The class of symmetric gauge functions are closed under taking the dual.

    \begin{prop}[Closure under the dual]
        If $\Phi:\reals^d \to \reals_+$ is a symmetric gauge function then so is $\Phi^*$.
    \end{prop}

\begin{proof}
    We simultaneously prove the symmetry and gauge invariance property of $\Phi^*$. To this end, let $P_d \circ \mathbb{Z}_2$ be the group of permutations and sign changes of the coordinates.
    Since $\Phi$ is a gauge function we have
    \[
    \{y \in \reals^d: \Phi(y) \leq 1\} = \{y \in \reals^d: \Phi(Py) \leq 1\} \qquad P \in P_d \circ \mathbb{Z}_2.
    \]
    Moreover, observe that $| \langle Px, Py \rangle | = | \langle x, y \rangle |$ for all $x,y \in \reals^n$ and all $P \in P_d \circ \mathbb{Z}_2$. Hence for all $P \in P_d \circ \mathbb{Z}_2$, we have 
    \begin{align*}
        \begin{split}
            \Phi^*(Px) &= \sup\left\{ |\langle Px, y \rangle| : \Phi(y) \leq 1 \right\}
            \\&= \sup \left\{ | \langle Px, y \rangle | : \Phi(P^{-1}y) \leq 1  \right\} 
            \\&= \sup \left\{ | \langle Px, P\tilde{y} \rangle | : \Phi(\tilde{y}) \leq 1  \right\} \qquad \text{(Change-of-variables $\tilde{y} \leftarrow P^{-1}y$)}
            \\&= \sup \left\{ | \langle x, \tilde{y} \rangle | : \Phi(\tilde{y}) \leq 1  \right\}
            \\ &= \Phi^*(x).
        \end{split}
    \end{align*}
    It is a standard fact that the dual of a norm is a norm. Hence $\Phi^*$ is indeed a symmetric gauge function. 
\end{proof}

The following corollary illustrates that the family of $\ell_p$ norms are closed under the $\ell_p$ transformations. 

 \begin{cor}[Exercise IV.1.9~\citep{bhatia97}]
    Let $\Phi$ be a symmetric gauge function and let $p \geq 1$. Let
$$
\Phi^{(p)}(x)=\left[\Phi\left(|x|^p\right)\right]^{1 / p} 
$$
then $\Phi^{(p)}$ is a symmetric gauge function.
Suppose $\Phi_p$ denotes the family of $\ell_p$ norms, then 
        \[
        \Phi_{p_1}^{(p_2)} = \Phi_{p_1 p_2} \qquad \forall p_1, p_2 \geq 1.
        \]
\end{cor}

\section{Structured Regularization}
We consider regularizations of problem~\ref{eq:prob} of the form
\begin{equation}
    \min_{X \in \pd} \hat{\phi}(X) \defas \phi(X) + R(X) \; ,
\end{equation}
where the regularizer $R:\pd \to \reals$ is determined by the structure of the constraints $\Xc$ and the objective $\phi$. 
In this section, we will show a construction of $R$ via symmetric gauge functions for two types of constraints and examine the properties of the regularized objective $\hat{\phi}$. We will further discuss the construction of a wider range of regularizers following similar ideas.

\paragraph{Sparsity regularization} In the first instance of problem~\ref{eq:prob} that we consider $\Xc$ induces sparsity in the eigenspectrum of the solution, i.e., it encourages low-rank solutions. The corresponding \emph{sparsity regularization problem} is of the form 
\begin{equation}\label{eq:sparsity_reg_problem}
    \hat{\phi}(X) \defas \phi(X) + \beta S(X)   \; ,
\end{equation}
where $S:\pd \to \reals$ and $\beta \geq 0$ a tunable hyperparameter. Problems of this form arise for instance in the computation of statistical estimator, such as covariance estimation~\citep{bien}, Gaussian graphical models~\citep{uhler2017gaussiangraphicalmodelsalgebraic}, which we discuss later in the paper, as well as in recommender systems~\citep{vandereycken2013low}.

\paragraph{Ball constraint regularization} In the second class of problems that we consider, side information on the solution is given by a coarse empirical estimate and we want to confine our optimization procedure within a neighborhood of this estimator.
Let $d(X, \hat{X}): \pd \times \pd \to \reals_+$ denote a metric on the  of $\pd$, $\hat{X} \in \pd$ some fixed nominal estimate of the optimum $X^* \defas \argmin_{X \in \pd}\phi(X)$, and $\beta \geq 0$ a tunable hyperparameter. Then the \emph{ball constraint regularization} problem is given by
\begin{equation}\label{eq:ball_constrain_problem}
    \hat{\phi}(X) \defas \phi(X) + \beta d(X, \hat{X}) \; .    
\end{equation}
Problems of this form arise in the computation of statistical estimators, e.g., the Karcher mean~\citep{Karcher1977RiemannianCO} and optimistic likelihood~\citep{Nguyen2019CalculatingOL}, both of which we discuss later in the paper. 

\paragraph{Outline} This section is structured as follows. We will motivate symmetric gauge functions as a natural starting point from which we can generate the sparsity and the ball constraint regularizers. For each class of regularizers we will first show that they arise from symmetric gauge functions. Second, we justify their value by showcasing their desirable algorithmic properties for g-convex optimization. Third,  we identify examples in which they arise in the prior optimization literature and show that rewriting these well-known regularizers in terms of our two classes of regularizers provides insight into their algorithmic properties. Fourth, we show that we can generate novel sparsity or ball constraint regularizers in a principled manner from these symmetric gauge functions. Moreover, leveraging the algebraic structure of symmetric gauge function, a wide range of new regularizers for other types of objectives and constraints can be constructed. Finally, we present a principled way of tuning the hyperparameters introduced by these regularizers. 

\subsection{Symmetric Gauge Functions as Sparsity and Ball Constraint Regularizers}
Let $\Phi \in \SG$ be a \emph{symmetric gauge function} (short: \emph{SG}) and let $\hat{X} \in \pd$ be fixed. In this section, we will show how to find the corresponding sparsity and ball constraint regularizers, i.e.,  $S_\Phi(X)$ and $d_\Phi(X, \hat{X})$, respectively.

\subsubsection{SG as Sparsity Regularizers}  

There is a deep connection between symmetric gauge functions and unitarily invariant norms by a theorem of von Neumann. We leverage this connection in order to find $S_\Phi(X)$.

    \begin{defn}[Unitarily Invariant Norms]
    We say that a norm $\| \cdot \|$ on $\reals^{d \times d}$ is  unitarily (or orthogonally) invariant if $\|UAV\| = \|A\|$ for all orthogonal operators $U,V$ on $\mathbb{R}^{d \times d}$ and all $A \in \mathbb{R}^{d \times d}$. 
    \end{defn}

The next theorem is von Neumann's~\citep{von1937some} result on the correspondence between symmetric gauge functions and unitarily invariant norms. We apply the result of Theorem IV.2.1~\citep{bhatia97} to the context of $\pd$.

    \begin{theorem}[von Neumann]
        Given a symmetric gauge function $\Phi$ on $\reals^d$, define a function on $\pd$ as 
        \[
        \|A\|_\Phi = \Phi(\lambda(A))
        \]
        where $\lambda(A)$ denotes the singular values of $A$. Then this defines a unitarily invariant norm on $\pd$. Conversely, given any unitarily invariant norm $\| \cdot \|$ on $\pd$, define a function on $\mathbb{R}^d$ by
$$
\Phi_{\| \cdot||}(x)=|| \operatorname{diag}(x)||,
$$
where $\operatorname{diag}(x)$ is the diagonal matrix with entries $x_1, \ldots, x_d$ on its diagonal. Then this defines a symmetric gauge function on $\mathbb{R}^d$.
    \end{theorem}

\begin{rmk}\normalfont
    Given a symmetric gauge function $\Phi \in \SG$, we propose using the corresponding unitarily invariant norm as our sparsity regularizer, that is, $S_\Phi(X) = \|X\|_\Phi$.
\end{rmk}

\begin{example}
    We provide two well-known examples of symmetric gauge function and their associated unitarily invariant norms. Let $A \in \pd$ and $\lambda(A) = \{\lambda_1, \ldots, \lambda_d\} \in \reals^d$ denote the ordered eigenvalues, i.e., $\lambda_1 \geq \cdots \geq \lambda_d > 0$. First, consider the class of \textit{Schatten p-norms} defined as 

\[
\begin{aligned}
    &\|A\|_p = \Phi_p(\lambda(A)) = \left( \sum_{j=1}^n |\lambda_j(A)|^p \right)^{\frac{1}{p}}, \qquad 1 \leq p < \infty,
    \\ &\|A\|_{\infty}=\Phi_{\infty}(\lambda(A))=\lambda_1(A)=\|A\| .
\end{aligned}
\]
Second, consider the class of \textit{Ky Fan k-norms} defined as 
\[
\|A\|_{(k)}=\sum_{j=1}^k |\lambda_j^{\downarrow}(A)|, \quad 1 \leq k \leq n
\]
which is the sum of the first $k$ largest eigenvalues of $A$.
\end{example}

\subsubsection{SG as Ball Constraints}\label{sec:SG_Ball_Constraints}
In addition, for every symmetric gauge function $\Phi \in \mathcal{SG}(\pd)$ there exists a complete metric $d_\Phi$ on the convex cone of $\pd$. These metrics can provide a way to encourage our iterates to stay within a neighborhood of an estimated solution.

We can define the length of a path $\gamma:[0,1] \to \pd$ with respect to $\Phi$ as follows.

    \begin{defn}
        For a path $\gamma:[0,1] \rightarrow \pd$ we define its length w.r.t a symmetric gauge function $\Phi: \mathbb{R}^n \rightarrow \mathbb{R}_{+}$ as
$$
L_{\Phi}(\gamma) \defas \int_0^1\left\|\gamma^{-1 / 2}(t) \gamma^{\prime}(t) \gamma^{-1 / 2}(t)\right\|_{\Phi} d t .
$$
    \end{defn}

This gives rise to the following metric $d_\Phi$ associated with $\Phi$.

    \begin{defn}
        For a given symmetric gauge function $\Phi: \mathbb{R}^d  \rightarrow \reals_+$ we define the distance between $A, B \in \pd$ with respect to $\Phi$ as
$$
d_{\Phi}(A, B) \stackrel{\text { def }}{=} \inf \left\{L_{\Phi}(\gamma): \gamma \text { is a path from } A \text { to } B\right\}.
$$
    \end{defn}

It turns out that $d_\Phi$ can be expressed in terms of the unitarily invariant norm $\|\cdot\|_\Phi$. Moreover $d_\Phi$ is a complete metric on the convex cone of $\pd$ with several nice properties illustrated by the following theorem.

    \begin{theorem}[Theorem 2.2~\citep{LIM2013115}]\label{theorem:dphi_properties}
     We have $d_{\Phi}(A, B)=\left\|\log \left(A^{-1 / 2} B A^{-1 / 2}\right)\right\|_{\Phi}$ and $d_{\Phi}$ is a complete metric distance on the convex cone of $\mathbb{P}_d$ such that for $A, B \in \mathbb{P}_d$ and for invertible matrix $M$,
\begin{enumerate}
    \item $d_{\Phi}(A, B)=d_{\Phi}\left(A^{-1}, B^{-1}\right)=d_{\Phi}\left(M A M^*, M B M^*\right)$;
    \item $d_{\Phi}(A \# B, A)=d_{\Phi}(A \# B, B)=\frac{1}{2} d_{\Phi}(A, B)$, where $A \# B=A \#_{\frac{1}{2}} B$;
    \item $d_{\Phi}\left(A \#_t B, A \#_s B\right)=|s-t| d_{\Phi}(A, B)$ for all $t, s \in[0,1]$;
    \item $d_{\Phi}\left(A \#_t B, C \#_t D\right) \leq(1-t) d_{\Phi}(A, C)+t d_{\Phi}(B, D)$ for all $t \in[0,1]$.
\end{enumerate}
    \end{theorem}

The following examples show that $d_\Phi$ recovers well-known metrics for specific symmetric gauge functions $\Phi$.

\begin{example}
    If we specify $\Phi$ to be the Schatten $2$-norm then the corresponding metric corresponds to the Riemannian affine-invariant metric on $\pd$. Moreover, if we choose $\Phi$ to be the $\infty$-Schatten norm then the associated metric is the \textit{Thompson metric}~\citep{thompson1963certain} on the positive definite cone. Namely,
    $$
d_{\infty}(A, B) = \max \{\log M(B / A), \log M(A / B)\}
$$
where $M(B / A) \defas \inf \{\alpha>0: B \leq \alpha A\}=\lambda_1\left(A^{-1 / 2} B A^{-1 / 2}\right)=\lambda_1\left(A^{-1} B\right)$. 
\end{example}

\subsection{Properties of Symmetric Gauge Function Regularizers}\label{section:Algo_benefit_SG}
In this section we discuss useful properties of general symmetric gauge functions, and our regularizers $S_\Phi(X)$, $d_\Phi(X,\hat{X})$ in particular. We will discuss the algorithmic implications of these results below in sec.~\ref{sec:4.1}.

\subsubsection{Preserving g-convexity}
The following proposition shows that unitarily invariant norms are g-convex on $\pd$ which implies that our regularized objective (Eq.~\ref{eq:sparsity_reg_problem}) remains g-convex. We will need the following definition.

\begin{defn}
    Let $x, y \in \reals^d$. We say \emph{$x$ is weakly-submajorized by $y$} if 
    \[
    \sum_{j=1}^k x_j^{\downarrow} \leq \sum_{j=1}^k y_j^{\downarrow}, \quad 1 \leq k \leq d.
    \]
    We denote this by $x \prec_w y$.
\end{defn}

    \begin{prop}[Unitarily Invariant Norms are g-convex]
        Let $\Phi: \reals^d \to \reals$ be a symmetric gauge function. Then the unitarily invariant norm $\|\cdot\|_\Phi: \pd \to \reals_+$ defined by $\|\cdot\|_\Phi = \Phi(\lambda(A))$ is g-convex.
    \end{prop}

\begin{proof}     
    To show $\|\cdot\|_\Phi$ is g-convex it suffices to verify midpoint g-convexity. We use the notation $\lambda^\downarrow(A) \preceq \lambda^\downarrow(B)$ to denote $\lambda_j(A) \leq \lambda_j(B)$ for $j = 1, \ldots, d$ for the spectrum ordered in decreasing order, i.e., 
    \[
    \lambda_1(A) \geq \lambda_2(A) \geq \cdots \geq \lambda_d(A) \qquad \text{and} \qquad \lambda_1(B) \geq \lambda_2(B) \geq \cdots \geq \lambda_d(B) \; .
    \]
    It is known that symmetric gauge functions are monotone \citep{bhatia97}, that is, if $\lambda^\downarrow(A) \preceq \lambda^\downarrow(B)$  then 
    \[
    \|A\|_\Phi = \Phi(\lambda^\downarrow(A)) \leq \Phi(\lambda^\downarrow(B)) = \|B\|_\Phi \; ,
    \]
    where the equalities follow from the permutation invariance property of $\Phi$.
    For $A, B \in \pd$ the weighted geometric mean satisfies~\cite[Exercise 6.5.6]{bhatia07positivedefinitematrices}: 
    \[
    A \#_t B \preceq (1-t) A+t B \qquad \text { for } t \in[0,1].
    \]
    Recall that if $A \succeq B$ then $\lambda^\downarrow(A) \succeq \lambda^\downarrow(B)$ which follows from the min-max theorem:

    \[
    \lambda_k(A)=\min _{\substack{U \subset \mathbb{C}^n \\ \operatorname{dim}(U)=k}} \max _{x \in U \backslash\{0\}} \frac{x^{\top} A x}{x^{\top} x} \geq \min _{\substack{U \subset \mathbb{C}^n \\ \operatorname{dim}(U)=k}} \max _{x \in U \backslash\{0\}} \frac{x^{\top} B x}{x^{\top} x}=\lambda_k(B),
    \]
    for $k \in [d]$ where the inequality follows from the fact that $A \succeq B \implies A-B \succeq 0$, that is $x^\top (A-B) x \geq 0$ for all vectors $x \neq 0$.
    Hence setting $t=1/2$ in the geometric mean we have 
    \begin{equation}\label{eq:geomean_ineq}
        A \# B \preceq \frac{A + B}{2} \implies \lambda^\downarrow(A \# B) \preceq \lambda^\downarrow \left( \frac{A+B}{2}\right).
    \end{equation}
    Thus using \eqref{eq:geomean_ineq} and applying the monotonicity and permutation invariance of $\Phi$ we get
    \[
    \Phi(\lambda(A \# B)) = \Phi(\lambda^\downarrow(A \# B)) \leq \Phi\left( \lambda^\downarrow \left( \frac{A+B}{2}  \right) \right) = \Phi\left( \lambda \left( \frac{A+B}{2}  \right) \right). 
    \]
    Moreover, Exercise~II.1.14~\citep{bhatia97} implies 
    \begin{equation}\label{eq:submaj_AB}
        \lambda^\downarrow \left(\frac{A+B}{2}\right) \prec_w \lambda^\downarrow \left(\frac{A}{2}\right) + \lambda^\downarrow \left(\frac{B}{2}\right) \; .
    \end{equation}
    We further know that $\Phi$ satisfies the \textit{strongly isotone} property (see~\cite[page~45]{bhatia97}), i.e., 
    \[
    x \prec_w y  \implies \Phi(x) \leq \Phi(y) \qquad \forall x,y \in \reals_+^n.
    \]
    Applying permutation invariance of $\Phi$ with \eqref{eq:submaj_AB} gives
    \begin{equation}\label{eq:weak_maj}
        \Phi\left(\lambda\left(\frac{A+B}{2}\right) \right) = \Phi\left(\lambda^\downarrow \left(\frac{A+B}{2}\right) \right)  \leq \Phi\left(\lambda^\downarrow\left(\frac{A}{2}\right) + \lambda^\downarrow \left(\frac{B}{2}\right) \right).
    \end{equation}
    Finally, we have for all $A, B \in \pd$,
    \begin{align*}
        \begin{split}
            \|A \# B\|_\Phi &= \Phi(\lambda(A \# B)) 
            \\& \leq \Phi\left(\lambda\left( \frac{A+B}{2}\right)\right) \qquad \text{(Applying monotonicity of } \Phi \text{ to \eqref{eq:geomean_ineq}})
            \\& \leq \Phi\left(\lambda^\downarrow\left(\frac{A}{2}\right) + \lambda^\downarrow\left(\frac{B}{2}\right) \right) \qquad \text{(Apply \eqref{eq:weak_maj})}
            \\&= \Phi\left(\frac{1}{2}\lambda^\downarrow(A) + \frac{1}{2}\lambda^\downarrow(B)\right) \qquad (\text{Property of eigenvalues: }\lambda(A/2) = \frac{1}{2}\lambda(A))
            \\&\leq \frac{\Phi(\lambda^\downarrow(A)) + \Phi(\lambda^\downarrow(B))}{2} \qquad (\Phi \text{ is a norm; triangle inequ., pos. homogeneity})
    \\&=\frac{\Phi(\lambda(A)) + \Phi(\lambda(B))}{2} \qquad \text{(Remove $\downarrow$ by permutation invariance of $\Phi$)}
    \\&\defas \frac{\|A\|_\Phi + \|B\|_\Phi}{2} \; ,
        \end{split} 
    \end{align*}
    which proves the midpoint criterion for g-convexity.
\end{proof}
This implies that regularizing a g-convex objective with a symmetric gauge function (or its corresponding unitarily invariant norm) will maintain g-convexity. Moreover, since symmetric gauge functions are closed under positive scaling, we can always select a hyperparameter $\beta > 0$ to control the regularity of $\Phi$. 

\begin{rmk}
    We remark that functions that are convex with respect to the Euclidean metric may not be g-convex with respect to the Riemannian metric and vice-versa. For example, let $y_1, \ldots, y_n \in \reals^d$ be non-zero vectors and define the functions $f,g:\pd \to \reals$ by
    \[
    f(X) \defas \sum_{ij=1}^d |X_{ij}| \qquad \text{and} \qquad g(X) \defas \sum_{i=1}^n \log \left(y_i^\top X y_i\right).
    \]
    $f(X)$ is Euclidean convex since it is a norm but it is not g-convex with respect to the Riemannian metric (see~\cite[Proposition 11]{cheng2024disciplinedgeodesicallyconvexprogramming}). On the other hand, $g(X)$ is g-convex with respect to the Riemannian metric but not Euclidean convex (see~\cite[Proposition 12]{cheng2024disciplinedgeodesicallyconvexprogramming}).
\end{rmk}

\subsubsection{Properties of Sparsity Regularizers}
The following result shows that symmetric gauge functions are closed under a $\ell_p$\textit{-transformation}, which provides us with additional flexibility to tune its regularity. 
    \begin{prop}\label{prop:lp_sgf}
        Let $\Phi$ be a symmetric gauge function and $p \geq 1$. Define the $\ell_p$-transformation of $\Phi$ to be
        \[
        \Phi^{(p)}(x) \defas \left[ \Phi(|x|^p) \right]^{1/p}.
        \]
        Then $\Phi^{(p)}$ is also a symmetric gauge function.
    \end{prop}
\begin{proof}
    Gauge invariance and symmetry follows directly from the definition of $\Phi^{(p)}$. It remains to show that $\Phi^{(p)}$ is a norm. The absolute homogeneity and positive definiteness property of $\Phi^{(p)}$ follows directly from definition. The triangle inequality directly follows by applying  Minkowski's inequality for symmetric gauge functions~\cite[Theorem IV.1.8]{bhatia97}.
\end{proof}
Thus we can write Problem~\ref{eq:sparsity_reg_problem} as the following problem with more fine-grained regularity on the sparsity. For any $p \geq 1$ and $\beta > 0$, we can rewrite Problem~\ref{eq:sparsity_reg_problem} as
\begin{equation}\label{eq:structure_1}
    \hat{\phi}(X) = \argmin_{X \in \pd } \phi(X) + \beta \Phi^{(p)}(\lambda(X))    \; ,
\end{equation}
which is g-convex.  We note that if $\phi$ is g-convex and DC, then so is $\hat{\phi}$ in \eqref{eq:structure_1}. However, if $\phi$ is not DC then $\hat{\phi}$ is not necessarily DC. In that case, we may construct a symmetric gauge function regularizer to induce DC structure and sparse solutions (see Proposition~\ref{prop:diagonal_loading} below for an example).

\subsubsection{Properties of Ball Constraint Regularizers}
Similarly, we have the following desirable algorithmic properties of $d_\Phi$.

    \begin{prop}[Proposition 3.5~\citep{LIM2013115}]\label{prop:dphi_gcvx}
        Let $\Phi \in \SG$ be a symmetric gauge function and $d_\Phi: \pd \times \pd \to \reals$ be the corresponding metric. Then 
        \begin{enumerate}
            \item Every $d_\Phi$-ball is geodesically convex in $\pd$.
            \item The map $d_\Phi^\alpha(\cdot, Z): \pd \to \reals$ is geodesically convex for any $\alpha \geq 1.$
        \end{enumerate}
    \end{prop}

\noindent Proposition~\ref{prop:dphi_gcvx} implies that our ball constraint regularizer preserves desirable properties: For a g-convex and DC objective $\phi: \pd \to \reals$ and an appropriate choice of $\beta >0$ and $\alpha \geq 1$, the regularized problem 
\[
\argmin_{X \in \pd} f(X) + \beta d_\Phi^\alpha(X, Z) 
\]
is g-convex and DC, too. If DC structure is not present, one can use a regularizer that encodes the S-divergence, a function of symmetric gauge functions, to induce DC structure while retaining g-convexity (see Example~\ref{ex:sdiv} for more details).

\subsection{Designing new regularizers from symmetric gauge functions}\label{section:generating_reg_SG}
While the sparsity and ball constraints discussed above allow for regularizing many common constrained problem, the idea of leveraging symmetric gauge functions as regularizers applies more broadly. In this section we discuss how to build a rich class of regularizers for a wide range of constrained problems on the SPD manifold, using symmetric gauge functions and their convexity-preserving properties.

\subsubsection{Disciplined Geodesically Convex Programming with symmetric gauge functions}

\emph{Disciplined Geodesically Convex Programming} (short: DGCP) \citep{cheng2024disciplinedgeodesicallyconvexprogramming} is a framework for testing and verifying the geodesic convexity of nonlinear programs on Cartan-Hadamard manifolds.
DGCP presents a set of g-convex functions (\emph{atoms})  
and g-convex preserving operations (\emph{rules}) for the $\pd$ manifold. To test and verify convexity, a function is expressed in terms of atoms using only convexity preserving operations. Below, we discuss g-convex preserving operations on the symmetric gauge functions, which can be used to design new g-convex regularizers on $\pd$.

\paragraph{Rules}
\begin{prop}
    Let $S \subseteq \pd$ be a g-convex subset. Suppose the functions $f_i: S \to \reals$ are g-convex for $i= 1, \ldots, n$. Then the following functions are g-convex

    \begin{enumerate}
        \item $f(X) =  \max_{i \in \{1, \ldots, n\}} f_i(X)$
        \item $g(X) = \sum_{i=1}^n \alpha_i f_i(X)$ for $\alpha_1 \ldots, \alpha_n \geq 0$.
    \end{enumerate}
\end{prop}

The classes of positive linear and affine maps will be useful in identifying and constructing g-convex functions.

\begin{defn}[Positive Linear Map~\citep{bhatia07positivedefinitematrices}]
    A linear map $\Phi:\mathbb{P}_d \to \mathbb{P}_m$ is \textit{positive} when $A \succeq 0$ implies $\Phi(A) \succeq 0$ for all $A \in \mathbb{P}_m$. We say that $\Phi$ is \textit{strictly positive} when $A \succ 0$ implies that $\Phi(A) \succ 0$.
\end{defn}

\begin{defn}[Positive Affine Map~\citep{cheng2024disciplinedgeodesicallyconvexprogramming}]
    Let $B \succeq 0$ be a fixed symmetric positive semi-definite matrix and $\Phi: \pd \to \pd$ a positive linear operator. Then the function $\phi:\pd \to \pd$ defined by $\phi(X) \defas \Phi(X) + B $ is a \textit{positive affine operator}.
\end{defn}
Notably, strictly positive linear maps are g-convex.
\begin{prop}[Proposition 5.8~\citep{Vishnoi2018GeodesicCO}]\label{prop:strict_positive_linear}
    Let $\Phi(X)$ be a strictly positive linear operator from $\mathbb{P}_d$ to $\mathbb{P}_m$. Then $\Phi(X)$ is g-convex with respect to the Löwner order on $\mathbb{P}_m$ over $\mathbb{P}_d$ with respect to the canonical Riemannian inner product $g_X(U, V)\defas$ $\operatorname{tr}\left[X^{-1} U X^{-1} V\right]$. In other words, for any geodesic $\gamma:[0,1] \rightarrow \mathbb{P}_d$ we have that
$$
\Phi(\gamma(t)) \preceq(1-t) \Phi(\gamma(0))+t \Phi(\gamma(1)) \quad \forall t \in[0,1] \; .
$$
\end{prop}
We can further leverage the following g-convexity preserving compositions. 
\begin{prop}[Proposition 7~\citep{cheng2024disciplinedgeodesicallyconvexprogramming}]\label{prop:gcvx_affine_positive}

    Let  $\phi(X) \defas \Phi(X) + B$ where $\Phi(X)$ is a positive linear map and $B \succeq 0$. 
        Let $f: \pd \to \mathbb{P}_m$ be g-convex and monotonically increasing, i.e., $f(X) \preceq f(Y)$ whenever $X \preceq Y$. Then the function
        $g(X) \defas f\left( \phi(X)\right)$
        is g-convex.
\end{prop}

\begin{prop}[Proposition 2~\citep{cheng2024disciplinedgeodesicallyconvexprogramming}]\label{prop:cvx_gcvx_composition}
    Suppose $f: \pd \to \reals$ is g-convex. If $h: \reals\to \reals$ is nondecreasing and Euclidean convex, then $h \circ f$ is g-convex on $\pd.$
\end{prop}

Finally, the variable substitution $X = X^{-1}$ preserves g-convexity.

\begin{prop}[\cite{cheng2024disciplinedgeodesicallyconvexprogramming}]\label{lemma:inverse_gcvx}
Let $f: \pd \to \reals$ be g-convex.
Then $g(X) = f(X^{-1})$ is also g-convex.

\end{prop}

\paragraph{Atoms}
DGCP provides a fundamental set of atoms, which, when combined with its \textit{rules}, can be used to design novel regularizers for optimization on the $\pd$ manifold. 

\begin{example}[G-convex Atoms]\label{example:g_cvx_atoms}
    The following functions $\phi(X): \pd \to \mathbb{P}_m$ are g-convex with respect to the Riemannian affine invariant metric.
    \begin{enumerate}
    \item $\phi(X) = \log \det X$
        \item $\phi(X) = \tr(X)$ 
        \item If $S \succeq 0$ and $\phi(X) = \tr(SX)$
        
        \item Let $M \succeq 0$ and $M$ has no zero rows and $\phi(X) = M \odot X$ ($\phi$ is a strictly positive linear map)
        \item $\phi(X) = \Phi(\lambda(X))$ for any symmetric gauge function $\Phi(X): \reals^d \to \reals$
        \item Let $y \in \reals^{d}$ be a nonzero vector and $r \in \{-1, 1\}$. Then the function $\phi(X) = y^\top X^r y$ is strictly g-convex
        \item Let $h_i \in \reals^d$ be nonzero vectors for $i = 1, \ldots, n$ and $r \in \{-1, 1\}$. Then $\phi(X) = \log \left(\sum_{i=1}^n h_i^\top X^r h_i\right)$ is strictly g-convex.
    \end{enumerate}
\end{example}

\subsubsection{Designing new regularizers}
The utility of the DGCP framework, in conjunction with the class of symmetric functions, is now evident. By using the class of symmetric gauge functions $\SG$ along with the set of g-convex \textit{atoms}, and applying g-convex preserving \textit{rules}, we can design novel g-convex (and possibly DC) regularizers for optimization problems on $\pd$. Conversely, DGCP allows us to principally decompose known regularizers in terms of symmetric gauge functions and atoms to gain intuition of the properties that they induce. In short, DGCP and symmetric gauge functions presents a principled way of reasoning about regularization on $\pd$.

Below, we discuss a range of examples. Some of them recover known regularizers, reinterpreted via symmetric gauge functions. In particular, we will see that many classical regularizers can be written as compositions and transformations of the Schatten $1$-norm, which induces sparsity. We also introduce novel regularizers that have not been studied previously.

\paragraph{Known regularizers as symmetric gauge functions}

\begin{lem}[Regularizing Top-$k$ Eigenvalues]
      Let $\Phi:\pd \to \reals$ be the \emph{$k$-Ky-Fan norm} restricted to the set of positive definite matrices. For $k \in \{1, \ldots, d\}$,
    and $h(x) = |x|^p$ for $p \geq 1$. Then
    $R_\Phi(X) = \left | \sum_{j=1}^k \lambda^\downarrow_j(X) \right|^p$
    is g-convex. 
\end{lem}
This follows directly from Proposition~\ref{prop:cvx_gcvx_composition}.

\begin{prop}[Low-rank Regularizer]
    Let $\Phi:\pd \to \reals$ be the \emph{$n$-Ky-Fan norm} or the \emph{Schatten $1$-norm}. Then the low-rank inducing trace regularizer can be written as $R_\Phi(X)=\Phi(\lambda(X)) = \tr(X)$.
\end{prop}

Note that we can apply Proposition~\ref{prop:cvx_gcvx_composition} with $h(x) = \exp(x)$ to get another g-convex regularizer
    $R_\Phi(X) = \exp(\tr(X))$.

The following proposition illustrates two structure-inducing properties of the log-determinant barrier function $f(X) \defas \log \det X$, which is a g-linear (i.e. g-convex and g-concave) function. First, it encourages the iterates to be low-rank by penalizing the sum of the log of their eigenvalues. Second, it encourages the iterates to stay close to the identity matrix $I_d$ via the metric $d_\Phi(X, I_d)$.
\begin{prop}[Log-Det Barrier function]
     Let $\Phi$ be the Schatten $1$-norm and let $Z= I_d$ fixed. Then we have 
    \[
    \log \det X = d_\Phi(X, I_d)  = \sum_{j=1}^d \log \lambda_j (X).
    \]
\end{prop}

Next we introduce another class of sparsity-inducing regularizers, defined via Schatten $p$-norms.

\begin{prop}[Smooth Schatten $p$-Functions]\label{prop:smooth_schatten_norm}
 Define $R_\Phi^{p}$ as 
\[
R^{p}_\Phi(X) \defas \operatorname{Tr}\left(X+\gamma I\right)^{p / 2} = \sum_{i=1}^d \left(\lambda_i(X) + \gamma\right)^{p / 2} = \|X + \gamma I\|_\Phi^{p / 2}
\]
where $\Phi$ is the Schatten 1-norm.
This is known as the \textit{smooth Schatten $p$-function}~\citep{irls_fazel}. 
\end{prop}

The Schatten $p$-function is differentiable for $p > 0$ and Euclidean convex (short: e-convex) and g-convex for $p \geq 1$. For $p \in [0, 1)$, $R_\Phi^{(p)}$ is used in the iteratively reweighted least squares (IRLS-\textit{p}) algorithm to obtain low-rank solutions under affine constraints,  i.e., the affine rank minimization problem \citep{irls_fazel}. The level of sparsity is regulated by $p$ where stronger sparsity is induced for smaller values of $p$. 

\begin{rmk}\normalfont
Interestingly, if we take $p \to 0$, we obtain a familiar g-convex (in fact, g-linear) regularizer:

\[
\begin{aligned}
    \lim_{p \to 0}\frac{R^{(p)}_\Phi(X) - d}{p} &= \frac{1}{2}\sum_{i=1}^d \frac{\left[\left(\lambda_i(X) + \gamma \right)^{p / 2} - 1 \right]}{\frac{p}{2}}
    \\&= \frac{1}{2} \sum_{i=1}^d \log \left(\lambda_i(X) + \gamma \right) \qquad \Big( \lim_{p \to 0} \frac{x^p - 1}{p} = \log x \text{ for } x \in \reals_{++} \Big)
    \\&= \frac{1}{2}\log \det \left(X + \gamma I\right) \; ,
\end{aligned}
\]
which is used in the IRLS-0 algorithm in \citep{irls_fazel}.
\end{rmk}

The following lemma gives a useful connection between the S-divergence and the Schatten 1-norm.
\begin{lem}[S-Divergence]\label{ex:sdiv}
     Choose $\Phi$ to be the Schatten 1-norm. The S-divergence $\delta_S^2(X,Y): \pd \times \pd \to \reals_+$ is defined by 
\[
\delta_S^2(X,Y) \defas \log \det \left(\frac{X+Y}{2}\right) - \frac{1}{2}\log \det (XY)
\]
and can be expressed as 
\[
\delta_S^2(X,Y) = d_\Phi(X, X+Y) - \frac{1}{2}d_\Phi(X,Y) - \log 2 \|I_d\|_\Phi.
\]

\end{lem}

\begin{proof}
To see this, we perform simple algebra on the definition of $\delta_S^2(\cdot,\cdot)$:
\[
\begin{aligned}
    \delta_S^2(X,Y) &= \log \det X + \log \det \left(I_d + X^{-1}Y \right)  - n \log 2 - \frac{1}{2}\left[\log \det X + \log \det Y \right]
    \\&= \log \det \left(X^{-1} Y + I_d\right) - \frac{1}{2}\log \det X^{-1}Y - n\log 2
    \\&= \sum_{i=1}^d \left[ \log (1 + \lambda_i(X^{-1}Y) \right] - \frac{1}{2}\sum_{i=1}^d \lambda_i(X^{-1}Y) - n \log 2
    \\&= d_\Phi(X, X+Y) - \frac{1}{2}d_\Phi(X,Y) - \log2 \|I_d\|_\Phi \; ,
\end{aligned}
\]
where $\Phi$ is the Schatten 1-norm. 
\end{proof}

$\delta_S^2(\cdot, Y)$ has many desirable algorithmic properties. It is g-convex and DC, a metric, and its gradient can be computed particularly efficiently. It can be seen as a symmetrized version of the \textit{Log-Determinant Divergence.} The S-divergence is extensively studied in \citep{sra2013sdivergence}.

\begin{prop}[Diagonal Loading]\label{prop:diagonal_loading}
 We can sum the log-det barrier and the trace-inverse regularizer to get the \textit{diagonal loading} regularizer $R_\Phi(X): \pd \to \reals$ defined by 

$$R_\Phi(X) \defas \tr X^{-1} + \log \det X = \|X^{-1}\|_{\Phi} + d_{\Phi}(X,I_d) \; ,$$
where $\Phi$ is the Schatten 1-norm. Then $R_\Phi(X)$ is g-convex and DC in $X$.
\end{prop}
\begin{proof}
    Observe that $R_\Phi(X)$ is g-convex since it is a sum of a unitarily invariant norm $\|\cdot\|_\Phi$ and distance metric $d_\Phi(\cdot,\cdot)$ corresponding to a symmetric gauge function $\Phi$. The DC structure follows from the fact that $f(X) = \log \det (X)$ and $g(X) = \tr X^{-1}$ are concave and convex, respectively.
\end{proof}

This regularizer is also known as the \textit{shrinkage to identity} regularizer for covariance estimation \citep{structuredcovestimation_wiesel}. It is used when the number of samples is small relative to the number of features $d$. It encourages the solution to be close to the identity $I_d$ which is illustrated by the $d_\Phi(X, I_d)$ term. Another way to check this is to show that $R_\Phi(X)$ is minimized at $I$ by setting its gradient to zero.

\paragraph{Symmetric gauge functions in applications}
In many classical problems, the objective itself can be written  in terms of symmetric gauge functions without explicit regularizers. The techniques discussed above, including the DGCP framework, can be applied to these problems, too. We illustrate this on three examples.

\begin{example}[Square Root]\label{ex:sqrt}
     Sra~\citep{sra2015matrixsquareroot} introduced a formulation of the problem of
     computing the square root of $A \in \pd$ using the S-divergence (see Example~\ref{ex:sdiv}):
      \[
        \min _{X \in \pd} \left\{\phi(X) \defas \delta_S^2(X, A)+\delta_S^2(X, I) \right\}.
    \]
     In fact, this formulation provides a parametrization of the problem in terms of transformations and compositions of symmetric gauge functions.
     It has several desirable algorithmic properties, including DC structure and g-convexity. We will discuss this example in more detail below.
\end{example}

\begin{example}[Karcher Mean]\label{ex:karcher}
The Karcher mean problem \citep{pmlr-v202-weber23a} is the solution to the following problem.
Given data $\{A_1, \ldots, A_m\} \in \pd$ and $w \in \reals^m_+$ such that $\sum_{i=1}^m w_i = 1$ we solve
\[
\min_{X \in \pd} \left\{\phi(X) \defas \sum_{i=1}^m w_i \delta_R^2(X,A_i)\right\} \; ,
\]
where $\delta_R(X,A_i) \defas \| X^{-1/2}A_i X^{-1/2}\|_F$ is the Riemannian metric.
Sra~\citep{sra2013sdivergence} showed that the Karcher mean problem can be reformulated as
\[
\min_{X \in \pd} \left\{\phi(X) \defas \sum_{i=1}^m w_i \delta_S^2(X, A_i) \right\},
\]
where $\delta_S^2$ is again the S-divergence. The problem is g-convex and DC.
\end{example}

\begin{example}[Tyler's Estimator with Diagonal Loading]\label{example:tyler_estimator_diagonal_loading}
    Tyler's estimator \citep{tyler1987distribution} is a well-known robust covariance estimator. It is defined as the solution to the g-convex and DC  objective $\phi(\Sigma): \pd \to \reals$ defined by
    \[
    \argmin_{\Sigma \in \pd}\phi(\Sigma) \defas \frac{d}{n} \sum_{i=1}^n \log \left(x_i^T \mathbf{\Sigma}^{-1} x_i\right)+\log \det(\Sigma) \; ,
    \]
where $\{x_i\}_{i=1}^n \subseteq \reals^d$ are observed data samples. Wiesel and Zhang~\citep{structuredcovestimation_wiesel} introduce the penalty $R_\Phi(\Sigma) \defas \tr \left(\Sigma^{-1}\right)  + \log \det \Sigma$ which encourages a solution towards the identity matrix (see Proposition~\ref{prop:diagonal_loading}). By proposition \ref{prop:diagonal_loading} the regularizer is g-convex and DC. Hence
\[
\hat{\phi}(\Sigma) = \frac{d}{n} \sum_{i=1}^n \log \left(x_i^T \mathbf{\Sigma}^{-1} x_i\right)+\log \det(\Sigma) + \beta \left( \tr \left(\Sigma^{-1}\right) + \log \det \Sigma \right)
\]
is g-convex and DC.
\end{example}

\begin{example}[Normalized Regularized Tyler Estimator]\label{example:regularized_tyler_estimator} Suppose we are given $n$ independent realizations $\{x_1, \ldots, x_n\} \subseteq \reals^d$ drawn from an unknown zero mean distribution $f(x)$ with covariance $\Sigma$. To estimate $\Sigma$, one can use the \textit{normalized regularized Tyler estimator} (see~\cite[Section 4.2.3]{structuredcovestimation_wiesel}) which is (up to scaling) the minimizer  of 
\[
R_\Phi(Y) =\frac{d}{n} \sum_{i=1}^n \log \left[(1-\alpha) x_i^\top Y^{-1} x_i+\alpha \frac{\left\|x_i\right\|^2}{p} \operatorname{Tr}\left\{Y^{-1}\right\}\right]+\log \det Y \; ,
\]
where $\alpha \in [0,1]$ is a hyperparameter.
One can rewrite $R(Y)$ as 
\[
R_\Phi(Y) = d_\Phi(D, I_d) + d_\Phi(Y, I_d) \; ,
\]
where $\Phi$ is the Schatten $1$-norm and $D$ is the diagonal matrix with entries $$D_{ii} = \tr \left(Y^{-1}\left((1-\alpha) x_i x_i^\top + \alpha \|x_i\|^2 d^{-1} I_d \right) \right) \qquad i = 1, \ldots, d.$$
\end{example}

\section{Solving Structured Regularized Problems via CCCP}

\subsection{Exploiting the structural properties of $R_\Phi$}
\label{sec:4.1}
Recall that in Section~\ref{section:Algo_benefit_SG} we discussed that for any $\Phi \in \SG$ we can combine transformations of $S_\Phi$ and $d_\Phi$ to obtain a g-convex (and possibly DC) regularizer $R_\Phi$. 
Table~\ref{tbl:phi_hat_properties} illustrates the ``g-convex+DC'' property of the regularized problem $\hat{\phi}$ for varying g-convex and DC properties of $\phi$ and $R_\Phi$.

For completeness, we will provide a proof of these results.

\begin{prop}
    We prove the assertions in Table~\ref{tbl:phi_hat_properties}. Let $\phi: \pd \to \reals$ be the objective function, $R_\Phi(X):\pd \to \reals_+$ our generated regularizer for some $\Phi \in \SG$, and $\hat{\phi}:\pd \to \reals$ to be the resulting regularized objective. The following holds 
    \begin{enumerate}
        \item If $\phi(X)$ is g-convex and $R_\Phi(X)$ is g-convex then $\hat{\phi}(X)$ is g-convex.
        \item If $\phi(X)$ is g-convex and e-convex and $R_\Phi(X)$ is g-convex and DC then $\hat{\phi}(X)$ is g-convex and DC.
        \item If $\phi(X)$ is g-convex and DC and $R_\Phi(X)$ is g-convex and e-convex then $\hat{\phi}(X)$ is g-convex and DC.
        \item If $\phi(X)$ is g-convex and DC and $R_\Phi(X)$ is g-convex and DC then $\hat{\phi}(X)$ is g-convex and DC.
    \end{enumerate}
\end{prop}

\begin{proof}
Recall the structure of the regularized objective 
\[
\hat{\phi}(X) = \phi(X) + \beta R_\Phi(X) \qquad \text{for some } \beta > 0.
\]
\begin{enumerate}
    \item The g-convexity of $\hat{\phi}(X)$ follows from the fact that the sum of two g-convex functions is g-convex.
    \item The g-convexity of $\hat{\phi}(X)$ again follows from the fact that the sum of two g-convex functions is g-convex. Write $R_\Phi(X) = f(X) - h(X)$ where $f,h:\pd \to \reals$ are Euclidean convex functions. This implies DC. Moreover, we have
    \[
    \hat{\phi}(X) = \left(\phi(X) + \beta f(X)\right) - \beta h(X).
    \]
    Observe that $\phi(X) + \beta f(X)$ is e-convex and thus $\phi(X)$ is g-convex.
    \item Write $\phi(X) = f(X) - h(X)$ where $f,h:\pd \to \reals$ are e-convex. Then 
    \[
    \hat{\phi}(X) = \left( f(X) + \beta R_\Phi(X) \right) - h(X)
    \]
    which is g-convex and DC.
    \item Write $\phi(X) = f_1(X) - h_1(X)$ and $R_\Phi(X) = f_2(X) - h_2(X)$ where $f_k, h_k : \pd \to \reals$ are e-convex for $k = 1,2$. Write
    \[
    \hat{\phi}(X) = \left(f_1(X) + \beta f_2(X)\right) - \left(h_1(X) + \beta h_2(X)\right)
    \]
    which is clearly g-convex and DC.
\end{enumerate}
\end{proof}

\begin{rmk}\normalfont
    Table~\ref{tbl:phi_hat_properties} illustrates the interaction between the convexity and DC structure of $\phi(X)$ and $R_\Phi(X)$ and the resulting regularized problem $\hat{\phi}(X)$. In particular, due to the desirable properties of g-convex and DC, Table~\ref{tbl:phi_hat_properties} highlights the desirable properties of $R_\Phi(X)$, ranked from greatest to least: g-convex and DC, g-convex and e-convex, and g-convex. 
\end{rmk}

\begin{table}
    \centering
    \footnotesize
  \label{tbl:phi_hat_properties}
    \caption{
     Properties of $\hat{\phi}(X) = \phi(X) + \beta R_\Phi(X)$ where $R_\Phi$ is a regularizer consisting of (transformations of) $S_\Phi$ and $d_\Phi(X, \hat{X})$  for a specified $\Phi \in \SG.$ All regularized problems are g-convex but \textit{not} Euclidean convex.
    } 
  \begin{tabular}{c c c c c}
    \multicolumn{5}{c}{} \\
    \toprule
    \multicolumn{2}{c}{\textbf{Original}} & \multicolumn{2}{c}{\textbf{Regularized}} & 
      \multicolumn{1}{c}{\textbf{Example}} \\
    \cmidrule(r){1-2} \cmidrule(r){3-4} 
    $\phi(X)$ & $\mathcal{X}$ & $R_\Phi(X)$ & $\hat{\phi}(X)$ & {} \\
    \midrule
    g-cvx + DC, $\neg$ e-cvx  & $\pd$ & g-cvx + DC & g-cvx + DC & \makecell{Tyler's Estimator w/ Diagonal Loading \\(Ex.~\ref{example:tyler_estimator_diagonal_loading}, Sec. 4.2.2~\citep{structuredcovestimation_wiesel}) \\ \& Karcher Mean w/ S-Divergence (Ex.~\ref{ex:karcher})}\\ \\
    g-cvx, $\neg$ DC, $\neg$ e-cvx & $\pd$ & g-cvx & g-cvx & \makecell{Karcher Mean w/ Riemannian metric \\(Ex.~\ref{ex:karcher})} \\
    g-cvx+DC, $\neg$ e-cvx & 
    $B_{R}$ & g-cvx + DC & g-cvx + DC & \makecell{Optimistic Likelihood Problem \\(Sect.~\ref{section:optimisic_likelihood_problem})} \\
    g-cvx, $\neg$ DC, e-cvx & 
    $B_{R}$ & g-cvx + DC & g-cvx + DC & \makecell{Linear Regression on $\pd$ with S-divergence \\(Sect.~\ref{section:optimisic_likelihood_problem})} \\
    \bottomrule 
    \\
  \end{tabular}
\end{table}

\subsection{Complexity analysis}
In the previous section we discussed desirable properties that our class of regularizers induces in the reparametrized objective with respect to convexity and difference of convex structure (Tab.~\ref{tbl:phi_hat_properties}). We will now discuss the algorithmic implications of this induced structure, in particular, how it allows for applying CCCP algorithms in the constrained setting, as well as implications on the iteration and oracle complexities of the resulting optimization routine.

\subsubsection{Iteration complexity}
The \emph{iteration complexity} of an algorithm refers to the number of iterations required to reach an $\epsilon$-accurate solution and hence provides an important lens for non-asymptotic convergence analysis. Classical Riemannian first-order methods~\citep{frank-wolfe,liu2019simple} for constrained optimization on manifolds converge to the global optimum of problem~\ref{eq:prob} at a sublinear rate, whereas for Euclidean constrained methods, only sublinear convergence to a stationary point can be guaranteed. Our regularization approach, which reparametrizes problem~\ref{eq:prob} as an \emph{unconstrained} problem, further allows for leveraging CCCP (Alg.~\ref{alg:cccp}). The regularized objective fulfills the conditions of Theorem~\ref{prop:cccp-conv}, which guarantees an at most sublinear iteration complexity:
\begin{cor}
    Let $\hat{\phi}$ denote a regularization of problem~\ref{eq:prob} that is g-convex and DC. Then Alg.~\ref{alg:cccp} converges to a global optimum at a sublinear rate.
\end{cor}
Note that this result provides only an upper bound on the iteration complexity. We will see in the following section that, in some instances, faster, linear convergence can be obtained by leveraging special structure in the regularized objective to solve the CCCP oracle in closed form.

\subsubsection{Oracle complexity}
However, the iteration complexity alone does not provide a full characterization of the convergence rate. In order to understand the efficiency of an iterative algorithm, we further need to analyze the \emph{cost per iteration}. In the case of constrained Riemannian optimization, this cost is dominated by subroutines that enforce the constraints: In R-PGD~\citep{liu2019simple}, a projection onto the feasible region is computed, which can be costly. While R-FW~\citep{frank-wolfe} is projection-free, it requires a call to a linear oracle. This subroutine has closed form solutions or reductions to efficient oracles for specific domains and sets of constraints~\citep{frank-wolfe,scieur2023strong}, but is in general non-convex, which can present a challenge in practise. Since both approaches are \emph{Riemannian} algorithms, the respective subroutines, as well as the subsequent computation of the next iterate, require the implementation of Riemannian tools, such as exponential maps, Riemannian gradients and parallel transport operators. In contrast, CCCP requires only the implementation of Euclidean tools, which is often much faster. 

Like the Riemannian constrained optimization approaches, each CCCP iteration requires a call to a subroutine, the \emph{CCCP oracle} (see line 4 in Alg.~\ref{alg:cccp}). In some cases, it is possible to solve the minimization of the linear surrogate function therein in closed form, which renders the CCCP approach into a simple fixed-point algorithm. We will discuss examples of this form in the next section. Even if such a fixed-point algorithm cannot be derived, the oracle complexity is lower than that of the Riemannian approaches. First, the CCCP oracle is convex and no Riemannian tools are required to solve this subproblem numerically. Second, the structure of our class of regularizers, specifically that of our two canonical examples (ball and sparsity regularizers) provides additional algorithmic benefits, as we discuss below.

\subsubsection{Complexity of regularization}
Suppose our objective $\phi: \pd \to \reals$ is g-convex and DC, that is, $\phi$ is g-convex and can be written as $\phi(X) = f_1(X) + h_1(X)$ for convex functions $f_1, h_1: \pd \to \reals$. Moreover, suppose that $R:\pd \to \reals$ is a g-convex and DC regularizer that can be written as $R(X) = f_2(X) + h_2(X)$. In the case that $R(X)$ is convex then $h_2(X) = 0$.
Applying Algorithm~\ref{alg:cccp} to the structured regularized problem 
\[
\argmin_{X \in \pd} \left\{ \phi(X) + \beta R(X) = \left(f_1(X) + \beta f_2(X) \right) - \left(h_1(X) + \beta h_2(X)\right) \right\}
\]
requires solving the convex optimization problem 
\begin{equation}\label{eq:structured_cccp_oracle}
    \argmin_{X \in \pd} \left\{Q(X,Y) = \left(f_1(X) + \beta f_2(X) \right) - \left(h_1(Y) + h_2(Y) \right) - \langle \nabla (h_1 + h_2)(Y), X- Y \rangle  \right\}.
\end{equation}

In general, one can directly solve \eqref{eq:structured_cccp_oracle} with gradient descent, which amounts to computing the updates 
\begin{equation}\label{eq:}
\begin{gathered}
X_{\ell + 1} \leftarrow X_\ell - \eta \nabla_X Q(X_\ell, Y)
\\ \text{where} \qquad \nabla_X Q(X_\ell, Y) = \nabla (f_1 + f_2)(X_\ell) - \nabla (h_1 + h_2)(Y).
\end{gathered}    
\end{equation}
This requires computing the gradients $\nabla f_i$ and $\nabla h_i$ for $i = 1,2$ in each iteration. The complexity of the subroutine depends crucially on the number of calls to the corresponding gradient oracles. Choosing a regularizer that has efficiently computable gradients could allow for mitigating possible bottlenecks. We will discuss an example in the next section, where the ball constraint regularizer is parametrized via the S-divergence, which comes with numerical benefits for the gradient computation.

We contrast these considerations with the subroutines that handle constraints in R-PGD and R-FW. Both require calls to \emph{Riemannian} gradient oracles. This requires an additional projection, which introduces computational overhead (see sec.~\ref{sec:background}). Moreover, as discussed above, the subroutines generally require solving a nonconvex problem, which can be much more challenging than the CCCP oracle. An analysis of oracle complexities for certain sets of constraints in projection-free and projection-based methods can be found in~\citep{COMBETTES2021565,frank-wolfe,weber2021projection}.

\subsubsection{Ball constraint via S-divergence}\label{sec:ball_constraint_sdiv}

The computation of the Riemannian metric (Eq.~\ref{eq:riem-metric})
requires computing the generalized eigenvalues of $A$ and $B$, which introduces a computational bottleneck. To address this problem, \citep{efficientsimilarityCherian} introduced a \textit{symmetrized log-det based matrix divergence}, also known as the \textit{S-divergence} (Lem.~\ref{ex:sdiv}). Sra~\citep{sra2013sdivergence} discusses the relationship of the Riemannian metric $\delta_R$ and the S-divergence $\delta_S^2$ and its algorithmic implications. We present relevant properties of the S-divergence and its relation to the Riemannian metric $\delta_R$. 
\begin{prop}[Table~4.1 \citep{sra2013sdivergence}]\label{prop:properties_sdiv}
    Let $A,B, X \in \pd$. The S-divergence $\delta_S^2$ satisfies the following properties 
    \begin{enumerate}
        \item \textbf{Invariant Under Inversions.} $\delta_S\left( A^{-1},B^{-1}\right) = \delta_S(A,B)$
        \item \textbf{Invariant Under Conjugation.} $\delta_S(X^* A X,  X^*BX) = \delta_S(A,B)$
        \item \textbf{Bi-G-convex.} $\delta_S^2(X,Y)$ is g-convex in X,Y
        \item \textbf{Lower Bounded By Shifts.} $\delta_S^2(A+X, B+X) \leq \delta_S^2(A,B)$.
        \item \textbf{Geodesic As S-divergence. }$A \sharp B = \argmin_{X \in \pd} \delta_S^2(X,A) + \delta_S^2(X,B)$
    \end{enumerate}
    \end{prop}
Every property listed in Proposition~\ref{prop:properties_sdiv} is also satisfied by the Riemannian metric $\delta_R$, see~\citep[Table 4]{sra2013sdivergence} for more shared properties of $\delta_S^2$ and $\delta_R$. Moreover, we can relate the size of the metric balls induced by the $\delta_R$ and $\delta_S^2$ via the following proposition.
    \begin{prop}[Theorem 4.19~\citep{psra2013sdivergence}]\label{prop:sdivergence_Rdistance_bounds}
Let $A, B \in \pd$. Then we have $8 \delta_S^2(A, B) \leq \delta_R^2(A, B)$.
    \end{prop}
    \begin{prop}\label{prop:Rball_SBall}
        Fix $\hat{\Sigma} \in \pd$ and fix $\alpha>0$. Define the  sets 
        \begin{equation*}
            \mathcal{B}_R(\hat{\Sigma}; \alpha) \defas \{A \in \pd : \delta_R(A,\hat{\Sigma}) \leq \alpha\} 
        \end{equation*}
        and 
        \begin{equation*}
            \mathcal{B}_S(\hat{\Sigma}; \alpha) \defas \{A \in \pd : \delta_S^2 (A,\hat{\Sigma}) \leq \alpha\}.
        \end{equation*} 
        Then the subset-inequality
        \begin{equation*}
            \mathcal{B}_R(\hat{\Sigma}; \alpha) \subseteq \mathcal{B}_S(\hat{\Sigma}; C\alpha) 
        \end{equation*}
        holds for $C \geq \frac{\alpha}{8}$.
    \end{prop}

\begin{proof}
    By Proposition~\ref{prop:sdivergence_Rdistance_bounds}, we have the inequality 
    \begin{equation*}
        2 \sqrt{2} \delta_S(A,B) \leq \delta_R(A,B) \qquad \forall A, B \in \pd.
    \end{equation*}
   Let $\alpha>0$ and suppose $A \in \mathcal{B}_R(\hat{\Sigma}; \alpha)$. By definition and applying the inequality above, we have
    \begin{align*}
        \begin{split}
            \delta_R(A, \hat{\Sigma}) \leq \alpha &\implies 2 \sqrt{2}\delta_S (A,\hat{\Sigma}) \leq \alpha
            \\& \implies \delta_S^2(A, \hat{\Sigma}) \leq \frac{1}{8}\alpha^2 \; .
        \end{split}
    \end{align*}
    Hence $A \in \mathcal{B}_S(\hat{\Sigma}; C\alpha)$ for any $C \geq \frac{\alpha}{8}$. Since $A \in \mathcal{B}_R(\hat{\Sigma}; \alpha)$ was arbitrarily selected we have 
    \[
    \mathcal{B}_R(\hat{\Sigma}; \alpha) \subseteq \mathcal{B}_S(\hat{\Sigma}; C\alpha) \qquad \forall C \geq \frac{\alpha}{8}.
    \]
\end{proof}
This suggests that the $S$-divergence can be leveraged for an efficient relaxation of the ball constraint regularizer: Suppose we have an optimization problem constrained to lie within a Riemannian distance ball $\mathcal{B}_R(\cdot; \alpha)$ of radius $\alpha >0$. Since the S-divergence ball $\mathcal{B}_S(\cdot; C\alpha)$ is a superset of the Riemannian distance ball, we can replace the Riemannian distance ball with the S-divergence ball with radius $C \alpha$ for some $C \geq \alpha /8.$  This alludes to a more general relaxation technique which we discuss now.

\paragraph{Computational considerations}
Computing the S-divergence $\delta_S^2(A,B)$ requires 3 Cholesky factorizations for $A+B, A$ and $B$, whereas computing the Riemannian metric requires computing generalized eigenvalues at a cost of $4d^3$ flops for positive definite matrices. The cost gap between $\delta_S^2$ and $\delta_R$ only widens when considering their gradients 
\[
\begin{aligned}
& \nabla_A \delta_R^2(A, B)=A^{-1} \log \left(A B^{-1}\right) \\
& \nabla_A \delta_S^2(A, B)=(A+B)^{-1}-\frac{1}{2} A^{-1}.
\end{aligned}
\]
This is particularly well-illustrated in Table 2~\citep{pefficientsimilarityCherian}.
Hence, in many of our applications we will replace $\delta_R$ with the computationally advantageous $\delta_S^2$.

\paragraph{Relaxing the Riemannian Ball Constraint Problem} 
Let $\phi: \pd \to \reals$ be g-convex and DC. We would like to solve the following problem

\begin{equation}\label{eq:rball_constrain_problem}
    \begin{aligned}
        &X^* \defas \argmin_{X \in \pd} \phi(X) \qquad \text{subject to } \qquad \mathcal{B}_{R}(\hat{X}, r) 
        \\& \operatorname{where} \qquad \mathcal{B}_{R}(\hat{X}, r) \defas \{X \in \pd : {\delta_R}(X, \hat{X}) \leq r\}.
    \end{aligned}
\end{equation}
We use the notation $\hat{X} \in \pd$ to denote additional information as to where the true solution $X^*$ may lie and the radius $r>0$  corresponds to the confidence of such information. In light of Proposition~\ref{prop:Rball_SBall}
we can relax Problem~\ref{eq:rball_constrain_problem} by replacing the Riemannian ball $\mathcal{B}_R(\hat{X}, r)$ with the S-divergence ball $\mathcal{B}_S(\hat{X}, r/8)$. Then we can design the corresponding relaxed structured regularization problem as 

\begin{equation}\label{eq:rball_constrain_problem}
    \begin{aligned}
        &X^* \defas \argmin_{X \in \pd} \phi(X) + \beta \delta_S^2\left(X, \hat{X} \right)
    \end{aligned}
\end{equation}
for $\beta>0.$ We remark that Problem~\ref{eq:rball_constrain_problem} can be naturally applied to maximum likelihood problems \citep{Nguyen2019CalculatingOL}. In the experiments section, we will illustrate this relaxation procedure on the example of optimistic likelihood estimation.

\section{Applications}
In this section we illustrate our structured regularization framework on the square root problem, the Karcher mean problem, and the optimistic likelihood problem. In particular, we present Euclidean algorithms to solve for \emph{global} solutions of Euclidean nonconvex (but g-convex) optimization problems that have traditionally been viewed through a Riemannian lens.

\subsection{Square Root and Karcher Mean: Fixed-Point Solvers}
We have previously discussed structured relaxations of the Karcher mean (Example~\ref{ex:karcher}) and the square root (Example~\ref{ex:sqrt}) problems using the S-divergence. Since both objectives are g-convex and DC, applying the CCCP algorithm will provably converge to the optimal solution (Theorem~\ref{prop:cccp-conv}): G-convexity ensures that every local optimum is a global one and the CCCP algorithm will always converge to one such global optimum (see section~\ref{sec:background-cccp} for details).

Since the CCCP algorithm is a Euclidean algorithm, it circumvents the computational overhead introduced by Riemannian optimization tools. Moreover, the CCCP approach gives rise to fast fixed-point algorithms, whenever the CCCP oracle can be solved in closed form.
The following propositions establish such fixed-point approaches for the square root problem and the Karcher mean problem. Recall that the square root of a SPD matrix is given by the solution to the optimization problem
\begin{equation}\label{problem:sdiv_sqrt}
    \min _{X \in \pd} \phi(X) \defas \delta_S^2(X, A)+\delta_S^2(X, I) \; ,
\end{equation}
where the objective is composed of symmetric gauge functions. This formulation allows for the following effective CCCP approach:
\begin{prop}[Fixed-point approach for Square Root \citep{psra2015matrixsquareroot}]\label{prop:fp_sqrt}
    Applying the CCCP algorithm  to the problem of computing the square root of an SPD matrix is equivalent to iterating the fixed point map \begin{equation}\label{eq:sqrt_fp}
    X \leftarrow \left[(X + A)^{-1} + (X + I)^{-1}\right]^{-1} \; .    
    \end{equation}
\end{prop}
Recall that the Karcher mean problem~\citep{Karcher1977RiemannianCO,Jeuris2012ASA} is the solution to the following problem:
Given data $\{A_1, \ldots, A_m\} \in \pd$ and $w \in \reals^m_+$ such that $\sum_{i=1}^m w_i = 1$ we solve
\[
\min_{X \in \pd} \defas \sum_{i=1}^m w_i \delta_R^2(X,A_i) \; ,
\]
which can be rewritten as~\citep{sra2013sdivergence}
\begin{equation}\label{problem:sdiv_karcher_mean}
\min_{X \in \pd} \phi(X) \defas \sum_{i=1}^m w_i \delta_S^2(X, A_i)    \; .
\end{equation}
Then the following effective fixed-point approach can be derived:
\begin{prop}[Fixed-point approach for Karcher Mean \citep{pmlr-v202-weber23a}]\label{prop:fp_karcher}
  Applying the CCCP algorithm  to the Karcher mean problem is equivalent to the fixed-point algorithm 
    \begin{equation}\label{eq:karcher_fp}
        X \leftarrow \left[  \sum_{i=1}^m w_i \left( \frac{X+A_i}{2}\right)^{-1} \right]^{-1} \qquad k = 0 , 1, \ldots \; .
    \end{equation}
\end{prop}

\subsection{The Optimistic Likelihood Problem}\label{section:optimisic_likelihood_problem}
Next we consider the problem of computing optimistic likelihoods.
Let $\phi(x)$ denote the negative log-likelihood and suppose we want to incorporate additional side information on the solution $X^*$, such as an estimator $\hat{X} \in \pd.$ One natural formulation is the following constrained optimization problem:

\begin{equation}\label{problem:reg_additional_info}
    \begin{aligned}
        &X^* \defas \argmin_{X \in \pd} \phi(X) \qquad \text{subject to } \qquad \mathcal{B}_R(\hat{X}, r) 
        \\& \operatorname{where} \qquad \mathcal{B}_R(\hat{X}, r) \defas \{X \in \pd : \delta_R(X, \hat{X}) \leq r\} \; .
    \end{aligned}
\end{equation}
Note that $\phi(X)$ is g-convex and DC. The radius $r > 0$ can be interpreted a confidence interval for $\hat{X}$ with smaller values corresponding to higher confidence.
Incorporating such information via ball constraints can be advantageous in statistical problems, as we illustrate below.

The optimistic likelihood problem was first introduced by \citep{Nguyen2019CalculatingOL} for the special case where $\phi$ is the \emph{Gaussian} negative log-likelihood. In this section, we will develop an unconstrained formulation of~\eqref{problem:reg_additional_info} via structured regularization that can be solved efficiently using CCCP. 
In the subsequent sections we will show that our framework can extend beyond the Gaussian case to the class of Kotz distributions \citep{Fang2018}.

\subsubsection{Optimistic Gaussian Likelihood}\label{sec:GaussianMLE_intro}
 The \emph{Optimistic Gaussian Likelihood problem}~\citep{Nguyen2019CalculatingOL} concerns the following setting: Consider a set of i.i.d data points $x = \left(x_1, \ldots, x_n\right) \in \reals^d$ generated from precisely one of several Gaussian distributions $\{\mathcal{N}(0, \Sigma_c)\}_{c \in \mathcal{C}}$ with zero mean and covariance $\Sigma_c$ indexed by $c \in \mathcal{C}$ where $|C| < \infty$. Since we can parametrize the family of zero mean Gaussian distributions with $\pd$, i.e., $\mathcal{N}(0, \Sigma) \simeq \Sigma$, we denote the set of candidate distributions $\{\mathcal{N}(0, \Sigma_c)\}_{c \in \mathcal{C}}$ simply as $\{\Sigma_c\}_{c \in \mathcal{C}}$.
The goal is to determine the true Gaussian distribution by solving the following maximum likelihood problem over the class $\mathcal{C}$:
\begin{equation}\label{eq:optimistic_likelihood_naive}
    c^{\star} \in \underset{c \in \mathcal{C}}{\arg \min }\left\{\phi\left(\Sigma_c; x \right) \defas  \frac{1}{n} \sum_{k=1}^n x_k^{\top} \Sigma_c^{-1}x_k + \log \operatorname{det} \Sigma_c\right\} \; .
\end{equation}
In general, $\Sigma_c$ is unknown, but we can obtain an estimator $\hat{\Sigma}_c$ from the data. Moreover, Problem~\eqref{eq:optimistic_likelihood_naive} is highly sensitive to misspecification of the candidate distributions $\mathcal{N}(0, \Sigma_c)$. \citep{Nguyen2019CalculatingOL} addresses this by solving the following \textit{constrained g-convex} optimization problem for $c \in \mathcal{C}$:
\begin{equation}\label{eq:optimistic_likelihood_ball}
\min_{\Sigma \in \mathcal{B}_R(\hat{\Sigma}_c; \rho_c)} \phi( \Sigma; x) \qquad \text{where} \qquad  \mathcal{B}_R(\hat{\Sigma}_c; \rho_c) \defas \{\Sigma \in \pd: \delta_R\left(\Sigma, \hat{\Sigma}_c\right) \leq \rho_c\} \; .
\end{equation}
Problem~\eqref{eq:optimistic_likelihood_naive} has applications in machine learning (e.g. quadratic discriminant analysis \citep{discriminant_analysis_sapatinas}) and in statistics (e.g. Bayesian inference \citep{Price2017}). Problem~\eqref{eq:optimistic_likelihood_ball} is itself a natural formulation of maximum likelihood estimation with side information. That is, we want to compute a $\Sigma^*$ that maximizes the likelihood of observing the data $x_{1}, \ldots, x_{n}$ while staying close to $\hat{\Sigma}$ in order to leverage prior information.

The Riemannian metric on $\pd$ presents a natural notion of distance on densities that are parametrized by their covariance matrices. Among other desirable properties,  it satisfies desirable statistical properties such as invariances under congruences and affine transformations (see Proposition~\ref{prop:properties_sdiv} or Table 4.1 in~\citep{sra2013sdivergence}) 
\begin{equation}\label{eq:desirable_prop_Pd}
    \begin{aligned}
\delta_R\left(\Sigma^{-1}, \hat{\Sigma}^{-1}\right) &= \delta_R \left(\Sigma, \hat{\Sigma}\right) \\ \delta_R\left(A \Sigma A^{\top}, A \hat{\Sigma} A^{\top}\right)&=\delta_R\left(\Sigma, \hat{\Sigma}\right) \; ,
\end{aligned}
\end{equation}
for all invertible $A \in \reals^{d \times d}$ and $\hat{\Sigma}, \Sigma \in \pd$ \citep{riemanniangaussianSPD}.

In~\citep{Nguyen2019CalculatingOL}, the authors propose to solve Problem~\eqref{eq:optimistic_likelihood_ball} via R-PGD. 
However, their approach requires expensive computational operations such as inverse fractional powers of matrices arising from the projection onto $\mathcal{B}_R \subseteq \pd$. We present a relaxation of Problem~\eqref{eq:optimistic_likelihood_ball} that can be efficiently solved using the CCCP algorithm, circumventing  the need for such expensive computational machinery. 

Per the discussion of the computational advantages of the S-divergence $\delta_S^2$ (see Section~\ref{sec:ball_constraint_sdiv}) and the fact that it satisfies the desirable statistical properties of \eqref{eq:desirable_prop_Pd} (see Proposition~\ref{prop:properties_sdiv}) we can leverage Proposition~\ref{prop:Rball_SBall} to relax Problem~\eqref{eq:optimistic_likelihood_ball} as follows:
\begin{equation}\label{eq:sdiv_ball_MLE_Problem}
\begin{aligned}
    \argmin_{\Sigma \in \pd} \left\{ \phi(\Sigma) \defas \tr\left(S \Sigma^{-1}\right) + \log \det \Sigma  \right\} \qquad 
    \text{subject to} \quad \Sigma \in \mathcal{B}_{S}\left(\hat{\Sigma}; r \right) \; ,
\end{aligned}    
\end{equation}
where   $S = \frac{1}{N}\sum_{m=1}^N \left(x_m - \hat{\mu}\right)\left(x_m - \hat{\mu}\right)^\top$ denotes the sample covariance matrix and $\mathcal{B}_S\left(\hat{\Sigma};\rho\right) = \{X \in \pd : \delta_S^2(X, \hat{\Sigma}) \leq \rho \}$ is the \textit{S-divergence ball} with radius $\rho>0$ centered at $\hat{\Sigma}$.  Finally, performing structured regularization, we can rewrite the problem as
\begin{equation}\label{eq:sdiv_ball_MLE_Problem}
\begin{aligned}
    &\argmin_{\Sigma \in \pd} \left\{ \hat{\phi}(\Sigma) \defas \tr\left(S \Sigma^{-1}\right) + \log \det \Sigma  + \beta \delta_S^2\left(\Sigma, \hat{\Sigma}\right) \right\} 
\end{aligned}    
\end{equation}
for some hyperparameter $\beta >0.$ The objective $\hat{\phi}(\Sigma)$ is (strictly) g-convex and hence has a unique global minimum (if it exists). 
Solving this problem via CCCP yields Algorithm~\ref{alg:CCCP_on_MLE}. 
We argue that this approach can be used to find the global optimum.\\
\begin{cor}
    Initialize $\Sigma_0 \in \pd$ and denote $\{X_{\ell+1}\}$ as the iterates generated by Algorithm~\ref{alg:CCCP_on_MLE}. For a suitable choice of step sizes $\{\eta_{\ell}\}$, the sequence $X_{\ell+1}$ converges to the unique solution $\Sigma^*$ of Problem~\eqref{eq:sdiv_ball_MLE_Problem}.
\end{cor}

\begin{proof}
    Observe that $\hat{\phi}(0) = \phi(X)$ is undefined and $\phi(X)$ is continuous on $\pd$.  This implies the optimum is attained in $\pd$. Since $\hat{\phi}(\Sigma)$ is of g-convex and DC structure, the CCCP algorithm converges to the global minima \citep{pmlr-v202-weber23a}.
\end{proof}

\begin{algorithm}[H]
\small 
\begin{algorithmic}[1]
\State \textbf{Input:} $\Sigma_0, \hat{\Sigma} \in \pd$, $K, L \in \nat$, $\beta > 0$ and $\{\eta_\ell\} \subseteq \reals_{++}$
 \For{$k = 0, \ldots, K-1$}
 \State Precompute $\Sigma_k^{-1} + \beta \left( \Sigma_k + \hat{\Sigma} \right)^{-1}$ 
 \For{$\ell = 0, \ldots, L-1$}
    \State $\Sigma_{\ell+1} \leftarrow  \Sigma_\ell - \eta_\ell \left( -\Sigma_\ell^{-1} S \Sigma_\ell^{-1} - \frac{\beta}{2}\Sigma_\ell^{-1} + \Sigma_k^{-1} + \beta \left(\Sigma_k + \hat{\Sigma}\right)^{-1} \right) $
    \State Update $\Sigma_{k+1} \leftarrow \Sigma_{L}$
\EndFor
\EndFor
\State \textbf{Output:} $\Sigma_K$
\end{algorithmic}
 \caption{CCCP for Optimistic Gaussian Likelihood}\label{alg:CCCP_on_MLE}
\end{algorithm}
\subsubsection{Generalizing beyond the Gaussian setting}
In the previous section, we formulated Problem \eqref{eq:sdiv_ball_MLE_Problem} to find the maximum likelihood estimator of a \emph{Gaussian} density in the presence of side information. This formulation can be readily generalized to Kotz-type distributions \citep{Sra_conic_geometric_Opt_SPD, Fang2018} and the multivariate log-normal distribution.

\subsubsection{Kotz-type distributions}
We generalize the optimistic likelihood problem to a special class of \emph{elliptically contoured distributions}, also known as \emph{Kotz-type distributions}~\citep{Sra_conic_geometric_Opt_SPD,Fang2018}. For simplicity, we consider the case of zero-mean Kotz-type distributions.

If the density of a (zero-mean) Kotz-type distribution exists on $\reals^d$ then it is of the form 
\[
f(x;\Sigma) \propto \det (\Sigma)^{-1/2}g \big(x^\top \Sigma^{-1}x \big) \; ,
\]
where $\Sigma \in \pd$ and $g: \reals \to \reals_{++}$ is known as the \textit{density generator}.

\begin{example}
    The density generator $g(t) = \exp(-t/2)$ recovers the multivariate Gaussian distribution. 
\end{example}
In this section, we focus on the Kotz-type distribution that arise from specifying the density generator to be of the form
\[
g(t) = t^{\alpha - d/2} \exp\Bigg(-\Big(\frac{t}{b} \Big)^\beta \Bigg)
\]
for distribution parameters $\alpha \in (0, d/2], \beta, b >0$. The Kotz-type distributions with these parameters encompass many well-known instances, including Gaussian and  multivariate power exponential distributions, the multivariate-W distribution with shape parameter less than one, as well as the elliptical gamma distribution with shape parameter less than $d/2$, among others (see, e.g., sec. 5 in~\citep{Sra_conic_geometric_Opt_SPD} or \citep{Fang2018} for more examples).

Given $x_1, \ldots, x_n$ i.i.d observations sampled from a Kotz-type distribution, its negative log-likelihood assumes the form\footnote{We ignore any constants in the log-likelihood.} 

\begin{equation}\label{eq:kotz_nll}
    K(\Sigma; x_1, \ldots, x_n) = \frac{n}{2} \log \operatorname{det}(\Sigma)+\Big(\frac{d}{2}-\alpha\Big) \sum_{i=1}^n \log \big(x_i^T \Sigma^{-1} x_i\big)+\sum_{i=1}^n\Big(\frac{x_i^T \Sigma^{-1} x_i}{b}\Big)^\beta. 
\end{equation}
 In this section, we show that $K(\Sigma)$ given in \eqref{eq:kotz_nll} satisfies the g-convex and DC property for $\alpha \in (0, d/2]$ and for $b, \beta >0$. Moreover, we can find a global optimum (if it exists) using CCCP. The next two results provide sufficient conditions such that the maximum likelihood estimator (MLE) of \eqref{eq:kotz_nll} exists.

\begin{prop}[Lemma 44~\citep{Sra_conic_geometric_Opt_SPD}]\label{prop:kotz_existence_setting}
 Let the data $\mathcal{X}=\left\{x_1, \ldots, x_n\right\}$ span the whole space and for $\alpha<\frac{d}{2}$ satisfy
$$
\frac{|\mathcal{X} \cap L|}{|\mathcal{X}|}<\frac{d_L}{d-2 \alpha},
$$
where $L$ is an arbitrary subspace with dimension $d_L<d$ and $|\mathcal{X} \cap L|$ is the number of data points that lie in the subspace $L$. If $\left\|S^{-1}\right\| \rightarrow \infty$ or $\|S\| \rightarrow \infty$, then $K(S) \rightarrow \infty$.
\end{prop}

The sufficient condition given in Proposition~\ref{prop:kotz_existence_setting} can be satisfied in the underparametrized noisy setting. That is, the number of data points satisfies $n \leq d$ and they are perturbed by noise.

\begin{theorem}[Existence of MLE, Theorem 45~\citep{Sra_conic_geometric_Opt_SPD}]
     If the data samples satisfy Proposition~\ref{prop:kotz_existence_setting}, then the negative log-likelihood of a Kotz-type distribution~\eqref{eq:kotz_nll} has a minimizer (i.e., there exists an MLE).
\end{theorem}

The main result in the section is the following proposition.

\begin{prop}\label{prop:kotz_gcvx_dc}
    Let $\alpha \leq d/2$ and $b, \beta >0$ be fixed distribution parameters for a Kotz-type distribution. Then its negative log-likelihood given by \eqref{eq:kotz_nll} is g-convex and DC. 
\end{prop}

We remark that Proposition~\ref{prop:kotz_gcvx_dc} was first established in \citep{Sra_conic_geometric_Opt_SPD}. However, our proof of Proposition~\ref{prop:kotz_gcvx_dc} places an emphasis on establishing the g-convex and DC structure of the Kotz distribution via a DGCP analysis. We believe that such an approach will be fruitful in establishing g-convexity and DC structure for other objective functions.

The following lemmas will help us establish Proposition~\ref{prop:kotz_gcvx_dc}.

\begin{lem}\label{lemma:log_matrix_frac_gcvx}
    Let $\{x_1, \ldots, x_n\}\subseteq \reals^d \setminus \{0\}$ be fixed nonzero vectors.
    For $0 < \alpha \leq d/2$ the function $f:\pd \to \reals_{++}$ defined by 
    \[
    f(\Sigma) \defas \Big(\frac{d}{2}-\alpha\Big) \sum_{i=1}^n \log \left(x_i^T \Sigma^{-1} x_i\right)
    \]
    is Euclidean convex and (strictly) g-convex.
\end{lem}

\begin{proof}
    To establish g-convexity, we observe that the function $g_i(\Sigma) = \log (x_i^T \Sigma^{-1} x_i)$ is a strictly g-convex atom (see item 7 in Example~\ref{example:g_cvx_atoms}). Hence 
    \[
    f(\Sigma) =  \Big(\frac{d}{2}-\alpha\Big)  \sum_{i=1}^n g_i(\Sigma) 
    \]
    is strictly g-convex since conic combinations preserve (strict) g-convexity. 

    To establish Euclidean convexity of the component functions we express \(\sum_{i=1}^n \log(x_i^T A^{-1} x_i)\) as a log determinant of a product of matrices. Define a block diagonal matrix \(B\) such that
   \[
   B = \text{diag}(A^{-1}, A^{-1}, \ldots, A^{-1}) \; ,
   \]
   where each of the \(n\) blocks consists of the matrix \(A^{-1}\). Thus, \(B\) is a \(dn \times dn\) matrix. Now define a \(dn \times n\) matrix \(Y\) by stacking the vectors \(x_i\) in a specific block form:
   \[
   Y = \begin{bmatrix}
   x_1 & 0 & 0 & \cdots & 0 \\
   0 & x_2 & 0 & \cdots & 0 \\
   0 & 0 & x_3 & \cdots & 0 \\
   \vdots & \vdots & \vdots & \ddots & \vdots \\
   0 & 0 & 0 & \cdots & x_n
   \end{bmatrix}
   \]
   Here, each \(x_i\) is a \(d\)-dimensional column vector, and there are \(n\) such vectors. Consider the matrix product \(Y^T B Y\). This is an \(n \times n\) matrix where the \(i\)th diagonal element of \(Y^T B Y\) is given by:
   \[
   (Y^T B Y)_{ii} = x_i^T A^{-1} x_i
   \]
   and the off-diagonal elements are zero because of the block structure of \(Y\) and \(B\). Therefore,
  $Y^T B Y = \diag (x_1^T A^{-1} x_1, x_2^T A^{-1} x_2, \ldots, x_n^T A^{-1} x_n)$. Taking the logarithm of both sides, we get:
   \[
   \log(\det(Y^T B Y)) = \log\Big(\prod_{i=1}^n (x_i^T A^{-1} x_i)\Big) \; .
   \]
   Standard logarithm laws imply $\log(\det(Y^T B Y)) = \sum_{i=1}^n \log(x_i^T A^{-1} x_i)$.
We can prove the Euclidean convexity of the right hand side as follows:
\[
\begin{aligned}
    \sum_{i=1}^n \log(x_i^T A^{-1} x_i) &= \log\left(\det(Y^T B Y)\right)
    \\&=\log \left(\det(B) \det (YY^\top\right)
    \\&= \log \det (B) + \log \det Y^\top Y
    \\& = \log \det \left(\diag(A^{-1}, \ldots, A^{-1})\right) + \log \det \left( \diag(\|x_1\|^2, \ldots, \|x_n\|^2)\right)
    \\&= - \log \det \left(\diag(A, \ldots, A)\right) + C(x_1, \ldots,x_n)
    \\&= - \log \left(\det(A)\right)^n + C(x_1, \ldots,x_n)
    \\&= - n\log \det(A) + C(x_1, \ldots, x_n). 
\end{aligned}
\]
Since $A \mapsto - \log \det(A)$ is Euclidean convex on $\pd$ then it follows $f(A)$ is convex.
\end{proof}

\begin{lem}\label{lemma:matrix_frac_power_gcvx}
    Let $\{x_1, \ldots x_n\} \subseteq \reals^{d} \setminus \{0\}$ be nonzero vectors and fix $\beta > 0$. The function $g: \pd \to \reals_{++}$ defined by
    \[
    g(\Sigma) = \sum_{i=1}^n \left(\frac{x_i^\top \Sigma^{-1} x_i}{b }\right)^{\beta}
    \]
    is Euclidean convex and g-convex.
\end{lem}

\begin{proof}
    Without loss of generality we assume $b=1$ (since $b>0$ it can be absorbed by the vectors $x_i$). Define the functions 
    $h_i(\Sigma) := {x_i^\top \Sigma x_i}$ and $g(t) := t^{\beta}$.
    Observe that $h_i(\Sigma)$ is a g-convex atom (see Example~\ref{example:g_cvx_atoms}(6)) for $1 \leq i \leq n$, and in particular it is a positive linear map. Also $g(t)$ is g-convex and non-decreasing on the positive reals since for all $a,b >0$ and we have
    \[
    g\left(a \sharp_\theta b\right)=g\left(a^{1-\theta} b^\theta\right)=\left(a^\beta\right)^{1-\theta}\left(b^\beta\right)^\theta \leq(1-\theta) a^\beta+\theta b^\beta \quad \forall \theta \in[0,1] \; ,
    \]
    which follows from the (weighted) AM-GM inequality. Hence by Proposition~\ref{prop:gcvx_affine_positive} the function 
    \[
    g \circ h_i (\Sigma) = \left(x_i^\top \Sigma x_i \right)^\beta
    \]
    is g-convex. By applying Proposition~\ref{lemma:inverse_gcvx} the function 
    \[
    \psi_i(\Sigma) = g \circ h_i(\Sigma^{-1}) =  \left(x_i^\top \Sigma^{-1}x_i\right)^\beta
    \]
    is g-convex.    To establish Euclidean convexity, observe that 
    \begin{align*}
        g_i(\Sigma) &= \left( x_i^\top \Sigma^{-1}x_i \right)^{\beta} = \exp \bigg( \log \left( x_i^\top \Sigma^{-1}x_i \right)^{\beta }  \bigg)  \\
        &= \exp \bigg( \beta \log \left( x_i^\top \Sigma^{-1}x_i \right)  \bigg)  = \exp(\beta \ell_i(\Sigma)) \; ,
    \end{align*}
    where $\ell_i(\Sigma) = \log \left(x_i^\top \Sigma^{-1} x_i\right)$ is Euclidean convex by Lemma~\ref{lemma:log_matrix_frac_gcvx}.
     Since the exponential function $t \mapsto \exp(\beta t)$ is monotonically increasing and Euclidean convex on $\reals_{++}$ we have that
   $g_i(\Sigma) = \exp (\beta \ell_i(\Sigma))$ is convex for each $1 \leq 1 \leq n$ which follows by standard convex compositional rules. Thus we conclude that 
    \[
    g(\Sigma) = \sum_{i=1}^n \exp(\beta  
    \ell_i(\Sigma)) =\sum_{i=1}^n\left(\frac{x_i^T \Sigma^{-1} x_i}{b}\right)^\beta
    \]
    is Euclidean convex.
\end{proof}
We can now show Proposition~\ref{prop:kotz_gcvx_dc}.
\begin{proof}
The g-convexity of $K(\Sigma)$ follows from the fact that $\Sigma \mapsto \log \det \Sigma$ is g-linear and applying Lemmas~\ref{lemma:log_matrix_frac_gcvx},\ref{lemma:matrix_frac_power_gcvx}.
The DC structure of $K(\Sigma)$ follows from the fact that $\Sigma \mapsto \log \det (\Sigma)$ is concave. The second summand is convex by Lemma~\ref{lemma:log_matrix_frac_gcvx}. The third summand is convex by Lemma~\ref{lemma:matrix_frac_power_gcvx}. Hence the Kotz-type likelihood is a difference of convex functions.
\end{proof}
We remark that Section 5.4 of \citep{Sra_conic_geometric_Opt_SPD} provides a fixed-point algorithm to efficiently solve the problem $\argmin_{\Sigma \in \pd} K(\Sigma)$. However, we are interested in solving the corresponding optimistic likelihood problem
    \begin{equation}\label{eq:kotz_sdiv}
        \argmin_{\Sigma \in \pd} K(\Sigma) + \gamma \delta_S^2(\Sigma, \hat{\Sigma}) 
    \end{equation}
    for some fixed $\hat{\Sigma} \in \pd$ and $\gamma > 0$. We apply the CCCP algorithm, where we solve the surrogate minimization problem (CCCP oracle) via gradient descent (see Algorithm~\ref{alg:CCCP_on_Kotz}).  Note that Algorithm~\ref{alg:CCCP_on_Kotz} is structurally similar to Algorithm~\ref{alg:CCCP_on_MLE} where we replaced the sample covariance $S = \frac{1}{n}X X^\top$ with the matrix $X D_{\ell} X^\top$. Since the objective \eqref{eq:kotz_sdiv} is g-convex and DC, Algorithm~\ref{alg:CCCP_on_Kotz} converges to the global optimum whenever it exists.

\begin{algorithm}[H]
\small 
\begin{algorithmic}[1]
    \State \textbf{Input:} $\Sigma_0, \hat{\Sigma} \in \pd$, $K, L \in \nat$, $\alpha \in (0, d/2)$, $\beta, \gamma > 0$ and $\{\eta_\ell\} \subseteq \reals_{++}$
\State Set $D_\ell \defas \diag\left(\frac{\alpha - \frac{d}{2}}{x_i^\top \Sigma_\ell^{-1}x_i} - \frac{\beta}{b^\beta \left(x_i^\top \Sigma_\ell^{-1} x_i\right)^{\beta - 1}} : 1 \leq i \leq d\right)$.
 \For{$k = 0, \ldots, K-1$}
 \State Precompute $\frac{n}{2}\Sigma^{-1}_k + \gamma \left(\Sigma_k + \hat{\Sigma}_k\right)^{-1}$ 
 \For{$\ell = 0, \ldots, L-1$}
    \State $\Sigma_{\ell+1} \leftarrow  \Sigma_\ell - \eta_\ell \left(\Sigma_{\ell}^{-1}\left( X D_\ell X^\top \right)\Sigma_\ell^{-1} - \frac{\gamma}{2}\Sigma_\ell^{-1} + \frac{n}{2}\Sigma_k^{-1} + \gamma \left(\Sigma_k+\hat{\Sigma_k}\right)^{-1}
    \right)$
    \State Update $\Sigma_{k+1} \leftarrow \Sigma_{L}$.
    \EndFor
    \EndFor
\State \textbf{Output:} $\Sigma_K$ 
\end{algorithmic}
 \caption{CCCP on Optimistic Kotz Likelihood} \label{alg:CCCP_on_Kotz}
\end{algorithm}

\subsubsection{Multivariate T-Distribution}

In this section, we illustrate how one can apply Algorithm~\ref{alg:CCCP_on_Kotz} to a specific Kotz-type distribution. We focus on the $d$-dimensional mean zero multivariate $t$-distribution with $\nu >0$ degrees of freedom, parameterized by the scatter matrix $\Sigma \in \pd$. We denote this distribution as $\operatorname{MVT}(\Sigma; \nu)$. It has the density 
\[
f_X(x;\nu) = \frac{\Gamma[(\nu+d) / 2]}{\Gamma(\nu / 2) \nu^{d / 2} \pi^{d / 2}|{\Sigma}|^{1 / 2}}\left[1+\frac{1}{\nu}({x}-{\mu})^T {\Sigma}^{-1}({x}-{\mu})\right]^{-(\nu+d) / 2}.
\]
Its negative log-likelihood is proportional to $T_\nu(\Sigma)$ given by \\
\begin{align}
    T_\nu(\Sigma) &= \frac{n}{2}\log \det \Sigma + \frac{\nu + d}{2}\sum_{i=1}^n \log \Big( 1 + \frac{1}{\nu}x_i^\top \Sigma^{-1}x_i\Big) \\
    &\quad + \gamma \log \det \Big(\frac{\Sigma + \hat{\Sigma}}{2}\Big) - \frac{\gamma }{2}\log \det \big(\Sigma \hat{\Sigma}\big) \; . \nonumber
\end{align}
The multivariate $t$-distribution is a generalization of several well-known distributions: We recover the Student's $t$-distribution  by taking $d=1$ and $\Sigma = 1$. Taking $\nu \to +\infty$ recovers the multivariate normal density with mean 0 and covariance $\Sigma$. With $\nu = 1$ we recover the $d$-variate Cauchy distribution. Taking $\frac{\nu + d}{2}$ to be an integer recovers the $d$-variate Pearson type VII distribution \citep{Kotz2004}. 

Due to its ability to capture tail events (i.e., data involving errors with heavier than Gaussian tails), the multivariate t-distribution is often used for robust statistical modelling, especially when the assumptions of normality is violated~\citep{Lange1989,Nadarajah2005}. 

The corresponding optimistic likelihood problem with side  information $\hat{\Sigma}\in \pd$ can be written as
\begin{equation}\label{eq:multivariate_t_aux}
    \Sigma^* = \argmin_{\Sigma \in \pd} \left\{ \hat{\phi}(\Sigma) \defas  T_\nu(\Sigma) + \gamma \delta_S^2 \left(\Sigma, \hat{\Sigma}\right) \right\} \; ,
\end{equation}
where $\gamma > 0$ is a hyperparameter that denotes our confidence in $\hat{\Sigma}$. 
\begin{algorithm}[H]
\small 
\begin{algorithmic}[1]
    \State \textbf{Input:} $\Sigma_0, \hat{\Sigma} \in \pd$, $K, L \in \nat$, $\gamma, \nu > 0$ and $\{\eta_\ell\} \subseteq \reals_{++}$
 \For{$k = 0, \ldots, K-1$}
 \State Precompute $\gamma \left(\Sigma_k + \hat{\Sigma}_k\right)^{-1} + \frac{n}{2}\Sigma_k^{-1}$.
 \For{$\ell = 0, \ldots, L-1$}
    \State $\Sigma_{\ell+1} \leftarrow  \Sigma_\ell - \eta_\ell \Big(- \frac{\nu + d}{2} \Sigma_\ell^{-1} \big(\sum_{i=1}^n \frac{x_i x_i^\top}{\nu + x_i^\top \Sigma_\ell^{-1} x_i} \big)\Sigma_\ell^{-1}  - \frac{\gamma}{2}\Sigma_\ell^{-1} +  \gamma \big(\Sigma_k + \hat{\Sigma}_k\big)^{-1} + \frac{n}{2}\Sigma^{-1}_k\Big)$
    \State Update $\Sigma_{k+1} \leftarrow \Sigma_{L}$.
    \EndFor
\EndFor
\State \textbf{Output:} $\Sigma_K$
\end{algorithmic}
\caption{CCCP on Optimistic Multivariate T-Likelihood}\label{alg:aux_multivar_T}
\end{algorithm}
We solve problem \eqref{eq:multivariate_t_aux} using again a CCCP approach. To adapt Algorithm~\ref{alg:CCCP_on_Kotz} to this new setting, we can group the convex and concave terms of $\hat{\phi}(\Sigma)$ into the functions $f(\Sigma)$ and $g(\Sigma)$, respectively:
\[
\begin{aligned}
    f(\Sigma) &= \frac{\nu + d}{2}\sum_{i=1}^n \log \left(1 + \frac{1}{\nu}x_i^\top \Sigma^{-1}x_i\right) - \frac{\gamma}{2}\log \det \left(\Sigma \hat{\Sigma}\right) \\
    g(\Sigma) &= \frac{n}{2}\log \det \Sigma + \gamma \log \det \left( \frac{\Sigma + \hat{\Sigma}}{2} \right)
\end{aligned}
\]
Their gradients are computed as 
\[
\begin{aligned}
&\nabla f(\Sigma) = - \frac{\nu + d}{2} \Sigma^{-1} \left(\sum_{i=1}^n \frac{x_i x_i^\top}{\nu + x_i^\top \Sigma^{-1} x_i} \right)\Sigma^{-1}  - \frac{\gamma}{2}\Sigma^{-1}
\\&\nabla g(\Sigma) = \gamma \left(\Sigma + \hat{\Sigma}\right)^{-1} + \frac{n}{2}\Sigma^{-1}
\end{aligned}
\]
The convex CCCP surrogate function can be written as $$Q(\Sigma, \Sigma_k) = f(\Sigma) + \langle g(\Sigma_k), \Sigma - \Sigma_k \rangle$$ for convex $f$ and concave $g$. Thus to solve for $\argmin_{\Sigma \in \pd} Q(\Sigma, \Sigma_k)$ we can apply the gradient descent steps 
\[
\Sigma_{\ell+1} \leftarrow \Sigma_{\ell} - \eta_\ell \left( \nabla f(\Sigma_\ell) + \nabla g\left(\Sigma_k\right)\right) \; ,
\]
for a sequence of step-sizes $\{\eta_\ell\}$. The resulting approach is shown schematically in Algorithm~\ref{alg:aux_multivar_T}.

\subsubsection{Complexity Analysis}
We analyze the complexity of the discussed CCCP approaches on the example of Algorithm~\ref{alg:aux_multivar_T}. Each outer loop requires computing the term $\gamma \left(\Sigma_k + \hat{\Sigma}_k\right)^{-1} + \frac{n}{2}\Sigma_k^{-1}$ once. This requires two matrix inversions and one matrix addition. Each inner loop requires computing $\nabla f(\Sigma_\ell) + \nabla g(\Sigma_k)$, which can be computed as follows. First, construct a matrix $X = \left( x_1, \ldots, x_n \right) \in \reals^{d \times n}$. Second, compute $\Sigma_\ell^{-1}.$ Third, compute the vector $(\nu + x_i^\top \Sigma^{-1}x_i : i \in [n])$ by first computing the quantity 
\[
X \odot \left( \Sigma_\ell^{-1} X \right)
\]
and summing along its last dimension. Call the resulting vector $y$. This requires one matrix multiplication between $d \times d$ and $d \times n$ matrices, one Hadamard product, and $n^2$ scalar additions. For simplicity we can construct the diagonal matrix $D_\ell$ from the following vector
\[
D_\ell \defas \diag\left(\nu + x_i^\top \Sigma^{-1}x_i : i \in [n]\right) = \diag \left(\left(\nu \mathbf{1} + y\right)^{\odot -1} \right) \; ,
\]
where $\odot -1$ denotes element-wise inversion. Finally we obtain 
\[
\sum_{i=1}^n \frac{x_i x_i^\top}{\nu + x_i^\top \Sigma_\ell^{-1} x_i}  = X D_\ell X^\top \; ,
\]
which requires one matrix multiplication between two $n \times d$ matrices.
Then the gradient step requires two additional $d \times d$ matrix products and three matrix-matrix additions. Hence, each inner loop requires at most one matrix inversion and four matrix-matrix multiplications.

\subsubsection{Log Elliptically Contoured Distributions.}
A \textit{log elliptically contoured distributions} is a probability distribution whose logarithm lies in the class of elliptically contoured distributions. It turns out some log elliptically contoured distributions also satisfy the g-convex and DC structure. One example is the mean zero multivariate log-normal distribution $\text{LogMVN}(0, \Sigma).$ Given i.i.d observations $x_1, \ldots, x_n \in \reals^d_{++}$ from $\text{LogMVN}(0, \Sigma)$ its negative log-likelihood $\psi(\Sigma)$ is proportional to 
\begin{equation}\label{eq:log_norm_density}
\psi(\Sigma) \propto \frac{n}{2}\log \det \Sigma + \frac{1}{2}\sum_{j=1}^m \log (x_i)^\top \Sigma^{-1} \log (x_i)
\end{equation}
where we define the operation $\log x_i \in \reals^d$ to denote elementwise logarithm.

Observe that the g-convex and DC structure of \eqref{eq:log_norm_density} follows from the fact that $f(\Sigma) = \log \det \Sigma$ is Euclidean concave and g-linear and that the matrix fractional function $g_i(\Sigma) = \log (x_i)^\top \Sigma^{-1} \log (x_i)$ is Euclidean convex for each $1 \leq i \leq d.$ The g-convexity of $g_i(\Sigma)$ follows from it being a g-convex atom (see Example~\ref{example:g_cvx_atoms}(6)).

Hence one can also derive an algorithm like Algorithm~\ref{alg:CCCP_on_MLE} and Algorithm~\ref{alg:CCCP_on_Kotz} to solve the \emph{log-normal optimistic likelihood} optimization problem 
\[
\begin{aligned}
    &\argmin_{\Sigma \in \pd} \Psi(\Sigma) + \beta \delta_S^2(\Sigma, \hat{\Sigma})
    \\& \text{where} \quad \Psi(\Sigma) \defas \frac{n}{2}\log\det \Sigma + \frac{1}{2}\sum_{j=1}^m \log(x_i)^\top  \Sigma^{-1} \log (x_i)
\end{aligned}
\]
where $\beta >0$ is our regularization hyperparameter.

\subsection{Linear Regression on the $\pd$ manifold.}\label{section:linear_regression}
Our framework naturally encompasses linear regression on the manifold of positive definite matrices endowed with the affine invariant metric \citep{regression_pd}. 
Let $X \in \reals^{d \times d}$ be our data and $y \in \reals$ be our observed target. We aim to minimize the quadratic loss function $f(\hat{y}, y) = \frac{1}{2}\left(\hat{y}-y\right)^2$. For simplicity, we consider the model 
\begin{equation}
\hat{y} = f(W) = \tr(WX) \; , 
\end{equation}
wherein the least square problem becomes: 
\begin{equation}\label{eq:regression_pd}
\min_{W \in \pd} \frac{1}{2}\left(\tr(W \operatorname{Sym}(X)) - y\right)^2   
\end{equation}
with $\operatorname{Sym}(X) = \frac{X + X^\top}{2}$ (see~\cite[sec. 5]{regression_pd} for more details). The flexibility of our structured regularization framework allows us to regularize \eqref{eq:regression_pd} to incorporate side information. For instance, if we are given an estimator $\hat{W} \in \pd$ we can reformulate the problem as
\begin{equation}\label{eq:regression_aux}
    \min_{W \in \pd} \frac{1}{2}\left(\tr\left(W \operatorname{Sym}(X)\right) - y\right)^2 + \beta d_{\Phi}(W, \hat{W}) 
\end{equation}
for a specified symmetric gauge function $\Phi$ (see Section~\ref{sec:SG_Ball_Constraints}). Based on our discussions about its computational advantages and similarity to the Riemannian distance (see Section~\ref{sec:ball_constraint_sdiv}), one can replace $d_\Phi$ with the S-divergence, $\delta_S^2$, i.e., a regularizer based on symmetric gauge functions (see Example~\ref{ex:sdiv}).

Moreover, we can induce sparsity by adding a sparsity inducing regularizer $R_\Phi(W)$, obtaining
\begin{equation}\label{eq:regression_sparse}
    \min_{W \in \pd} \frac{1}{2}\left(\tr\left(W \operatorname{Sym}(X)\right) - y\right)^2 + \beta R_{\Phi}(W) \; .
\end{equation}
We note that since $d_\Phi$ and $R_\Phi$ are g-convex regularizers, the resulting optimization problems  \eqref{eq:regression_aux}  and \eqref{eq:regression_sparse} are both g-convex and hence can be solved using standard first-order Riemannian iterative methods. However, if one specifies $d_\Phi$ in \eqref{eq:regression_aux} to be the S-divergence regularizer
or replace $R_\Phi(W)$ with the diagonal loading regularizer (see Example~\ref{prop:diagonal_loading}) then the objective becomes g-convex \textit{and} DC.

One can adapt \eqref{eq:regression_aux} or \eqref{eq:regression_sparse} to the kernel learning or Mahalanobis distance learning problem (see sections 9.1 and 9.2 \citep{regression_pd}). For example, one can use \eqref{eq:regression_aux} to \emph{anchor} the solution to $\hat{W} \in \pd$ where $\hat{W}$ is an estimated Mahalanobis matrix from an auxiliary dataset that is representative of the data of interest. Alternatively, one can use \eqref{eq:regression_sparse} to encourage convergence to low-rank Mahalanbois matrices.

\section{Experiments}

\subsection{Square Root Problem}
We consider the problem of computing the square root of a matrix $A \in \pd$. 
 In this section, we demonstrate the competitive performance of CCCP against standard first-order Riemannian approaches for this problem. Recall that in this case, the CCCP oracle can be solved in closed form, rendering CCCP into a simple fixed point approach (see Eq.~\ref{eq:sqrt_fp}).

\paragraph{Data Generation}
We generate both medium-conditioned and ill-conditioned data. In the \textit{medium-conditioned} case, we construct 
\[
A = G G^\top \qquad \text{where} \qquad G_{ij} \stackrel{\text{i.i.d.}}{\sim} N(0,1).
\]
We further consider the Hilbert matrix
\[
H_{ij} = \frac{1}{i+j-1} \; ,
\]
a notable example of a very ill-conditioned matrix.

 \paragraph{Results}
 We showcase the performance of computing the square roots $A^{\frac{1}{2}}$ and $H^{\frac{1}{2}}$ with CCCP and benchmark against two Riemannian Conjugate Gradient (RCG) and Riemannian Gradient Descent (RGD).

Figures~\ref{fig: square_root_exp_medium} and \ref{fig: square_root_exp_ill} give results for all three algorithms in both cases. The stepsizes for RGD and RCG are determined by backtracking line search. In contrast, the fixed-point algorithm does not require a stepsize. As a reference point, the true optimum $X^*$ is computed using NumPy's square root function. The CCCP algorithm exhibits superior performance in terms of runtime. In the ill-conditioned case, we observe that RGD and RCG are unable to converge to the optimum. In contrast, CCCP attains the optimum even for ill-conditioned data. We believe this is due to the fact that the fixed-point algorithm computes the inverse of $X_k + A$ (medium-conditioned) rather than of $A$ (ill-conditioned), which is required by the gradient-based methods. This is illustrated in Figure~\ref{fig:cond_num_sqrt} and particularly evident, if we cross-compare the accuracy achieved by the three approaches:
\[
\begin{aligned}
    &\|H - \hat{H}_{\operatorname{CCCP}}\|_{F} = 8.9 \times 10^{-5}
    \\&\|H - \hat{H}_{\operatorname{RGD}}\|_{F} = 129.791
    \\&\|H - \hat{H}_{\operatorname{RCG}}\|_{F} = 126.335 
\end{aligned} \; .
\]
We discuss this observation in more detail in the next section.

\begin{figure}[htbp]
  \centering
  \begin{minipage}[b]{0.45\textwidth}
    \centering
    \includegraphics[width=\textwidth]{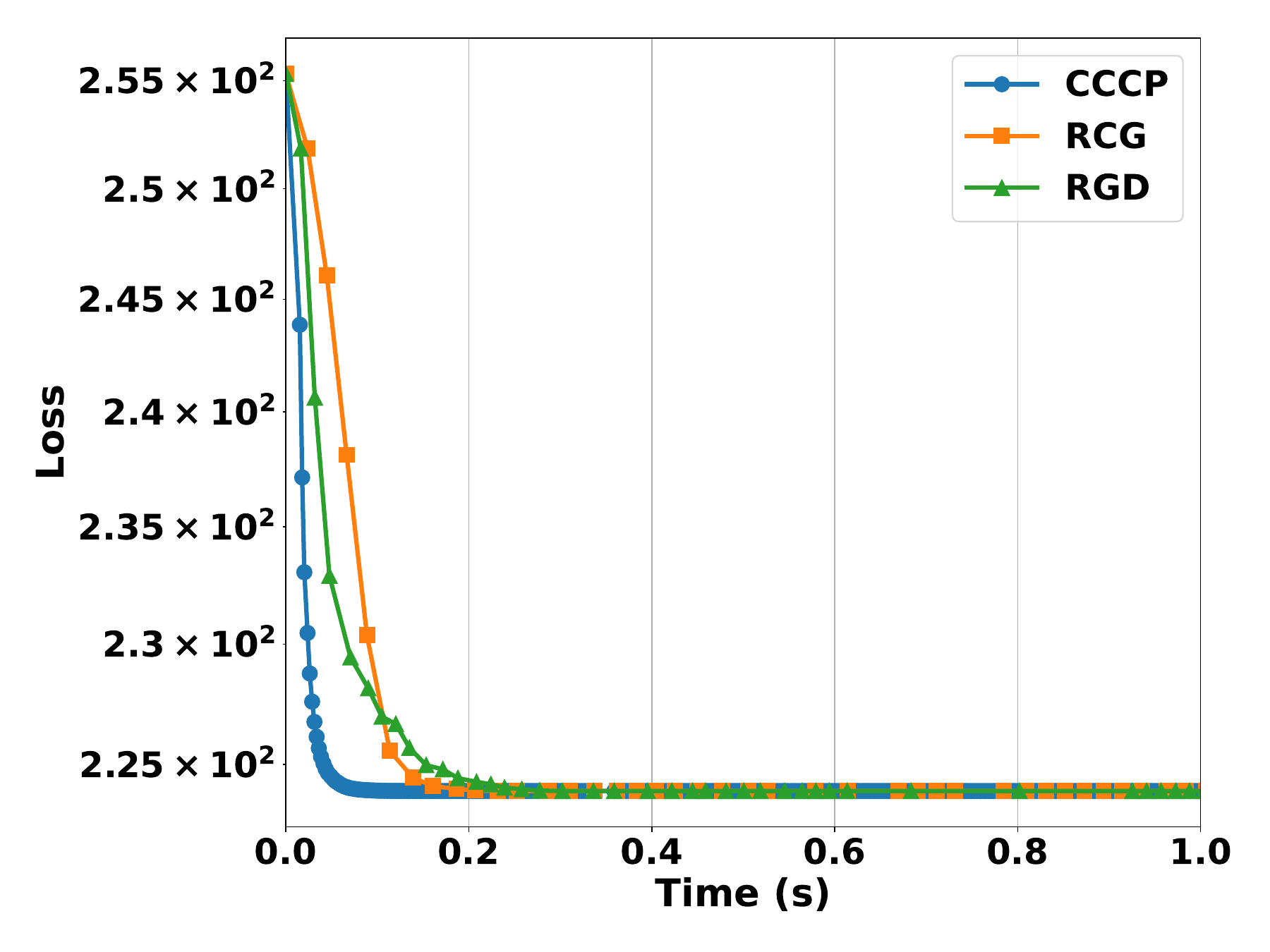}
  \end{minipage}
  \hfill
  \begin{minipage}[b]{0.45\textwidth}
    \centering
    \includegraphics[width=\textwidth]{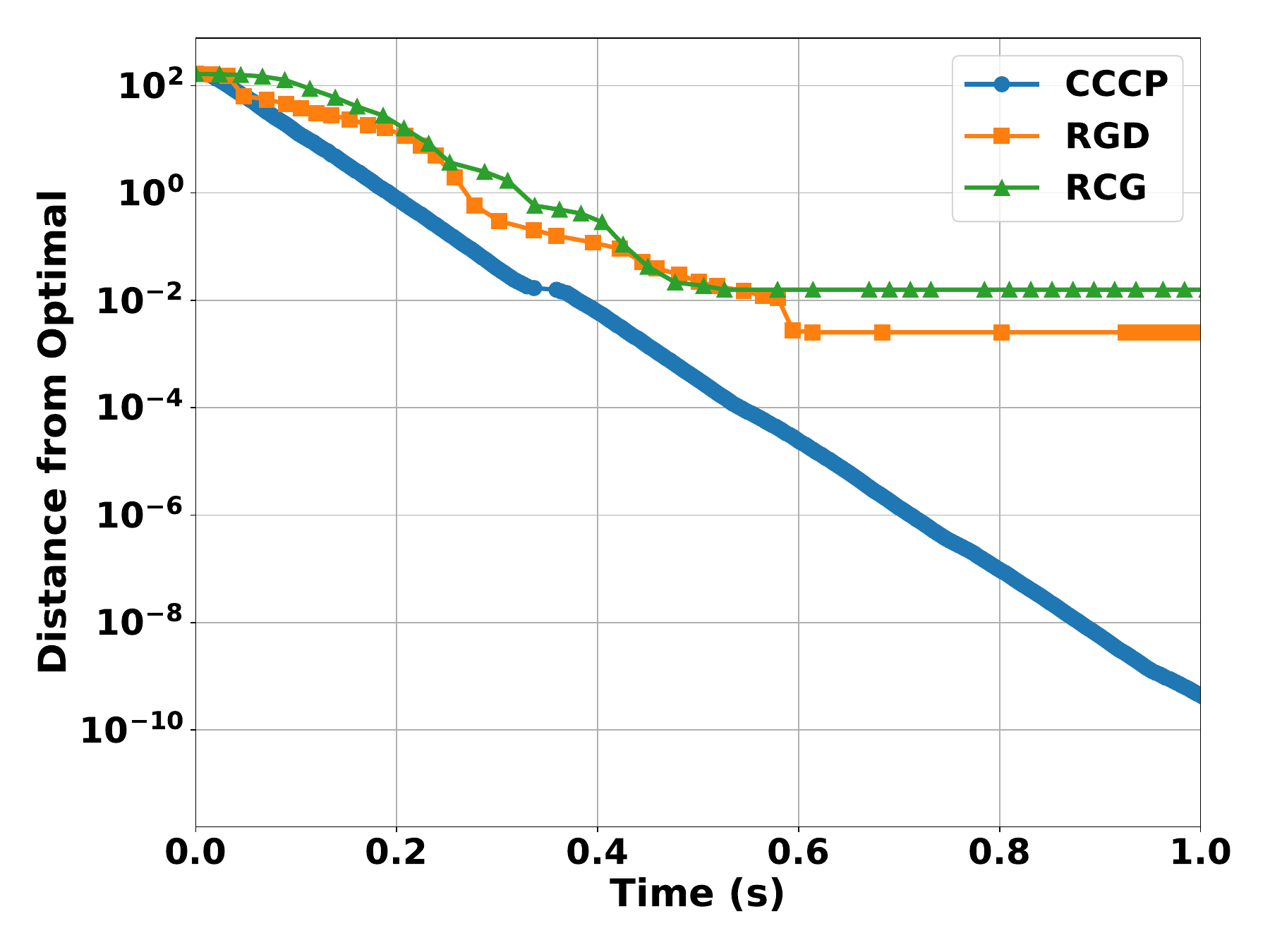}
  \end{minipage}
  \caption{We apply the fixed-point algorithm (see Proposition~\ref{prop:fp_sqrt}) to the medium conditioned $H \in \reals^{200 \times 200}$. We initialized all algorithms at $X_0 = 3 I_d$. Although all three methods are of the same order in terms of per-iteration-complexity, the fixed-point method exhibits superior runtime performance. The stepsizes for RGD and RCG are chosen using backtracking line search. In contrast, the fixed-point algorithm does not need a stepsize. Distance is measured in terms of the Frobenius norm.}
    \label{fig: square_root_exp_medium}
\end{figure}

\begin{figure}[htbp]
  \centering
  \begin{minipage}[b]{0.45\textwidth}
    \centering
    \includegraphics[width=\textwidth]{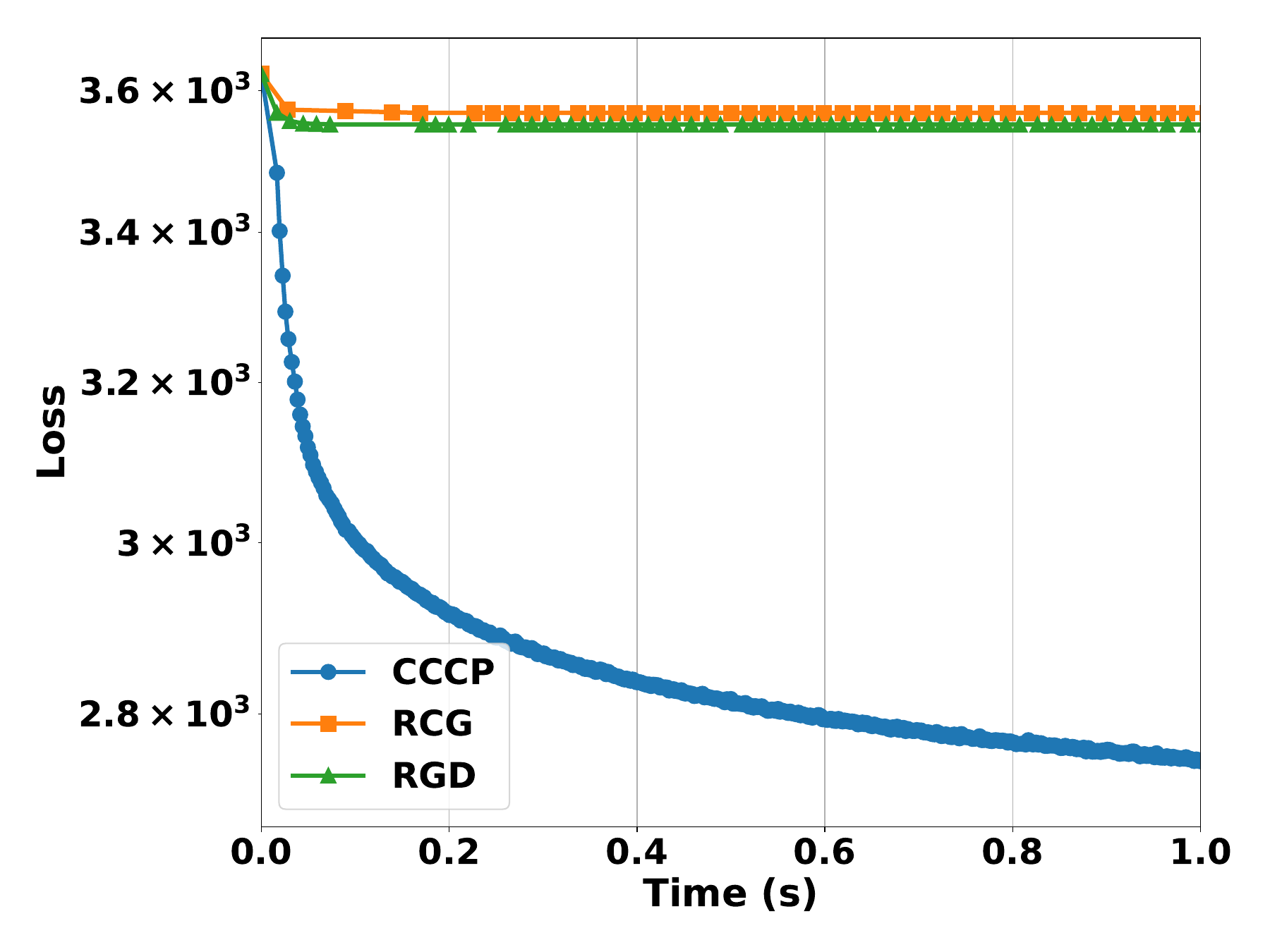}
  \end{minipage}
  \hfill
  \begin{minipage}[b]{0.45\textwidth}
    \centering
    \includegraphics[width=\textwidth]{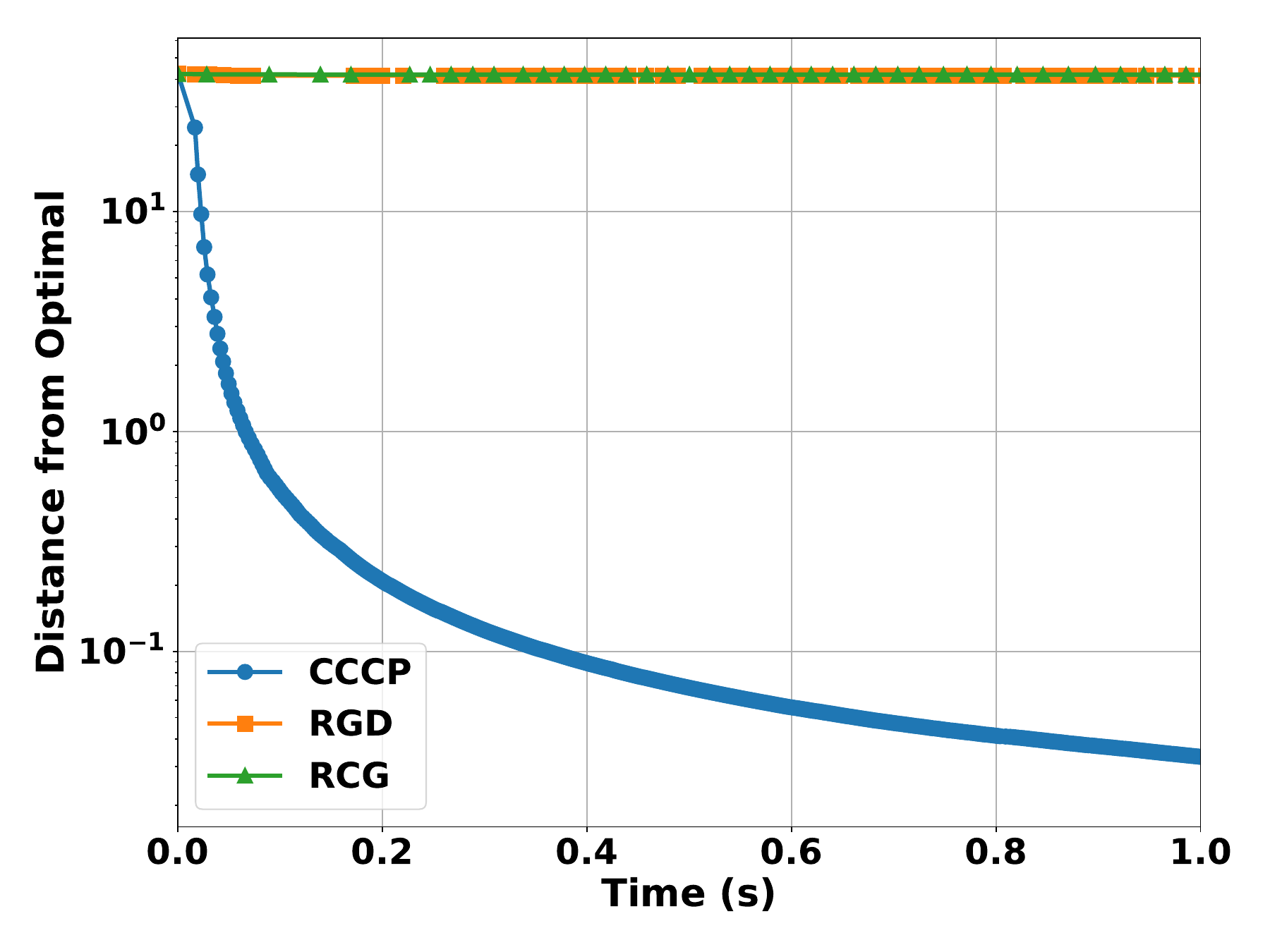}
  \end{minipage}
  \caption{We apply the fixed-point algorithm (see Proposition~\ref{prop:fp_sqrt}) to the ill-conditioned Hilbert matrix where we took dimension $d=200$. We initialized all algorithms at $X_0 = 3 I_d$. The benchmarks fail to converge whereas CCCP exhibits robustness to ill conditioning.}
  \label{fig: square_root_exp_ill}
\end{figure}

\paragraph{Gradient Steps}
The poor convergence of the Riemannian first-order methods holds across different very ill-conditioned matrices beyond the Hilbert matrix. For example, one can take the ill-conditioned linear low $k$-rank projections of $G G^\top$ with small perturbation $\delta I$ for $\delta \approx 0$. The first-order methods fail to converge for this ill-conditioned matrix as well.
In the following, we discuss this observation and a possible explanation on the example of the Hilbert matrix.
Recall that RGD preforms updates (for some stepsize $\eta > 0$)
\[
X \leftarrow \operatorname{Exp}(-\eta \operatorname{grad}\phi(X)) \; ,
\]
where the exponential map is defined as
\[
\operatorname{Exp}_X(t V)=X^{1 / 2} \exp \left(t X^{-1 / 2} V X^{-1 / 2}\right) X^{1 / 2}.
\]
Inserting this and the Riemannian gradient $\operatorname{grad}\phi(X)$ (see sec.~\ref{sec:background}) above gives
\[
\operatorname{Exp}(-\eta \operatorname{grad}\phi(X)) = X^{1/2} \exp \left(-\eta  X^{1/2} \nabla_X \Bar{\phi}(X) X^{1/2} \right)X^{1/2} 
\]
with
\[
\nabla \bar{\phi}(X) = \frac{1}{2}\left(\frac{X+A}{2}\right)^{-1} + \frac{1}{2}\left(\frac{X+I}{2}\right)^{-1} - X^{-1} \; .
\]
We suspect that computing the inverse of $X^{-1}$, i.e., inverting the Hilbert matrix at each iteration of the first-order methods results in numerical instability leading to the exhibited poor convergence behaviour. 
In contrast, the CCCP approach (Eq.~\ref{eq:sqrt_fp}) is robust to ill-conditioned matrices. Adding the identity $X+I$ improves the condition number of $X$ before taking inverses. Although $X+A$ does not have better conditioning than $A$ in general, we observe that in cases where $A$ is ill-conditioned, the matrix $X+A$ actually becomes well-conditioned in practice. Heuristically, this follows from Weyl's inequality which implies 
    \[
    \kappa(A+B) \leq \frac{\lambda_{\max}(A) + \lambda_{\max}(B)}{\lambda_{\min}(A) + \lambda_{\min}(B)} \qquad \text{for} \qquad A,B \in \pd \; ,
    \]
    where $\kappa(X)$ is the condition number of $X.$ This is demonstrated in Figure~\ref{fig:cond_num_sqrt}.
\begin{figure}[ht]
  \centering
  \begin{minipage}[b]{0.45\textwidth}
    \centering
    \includegraphics[width=\textwidth]{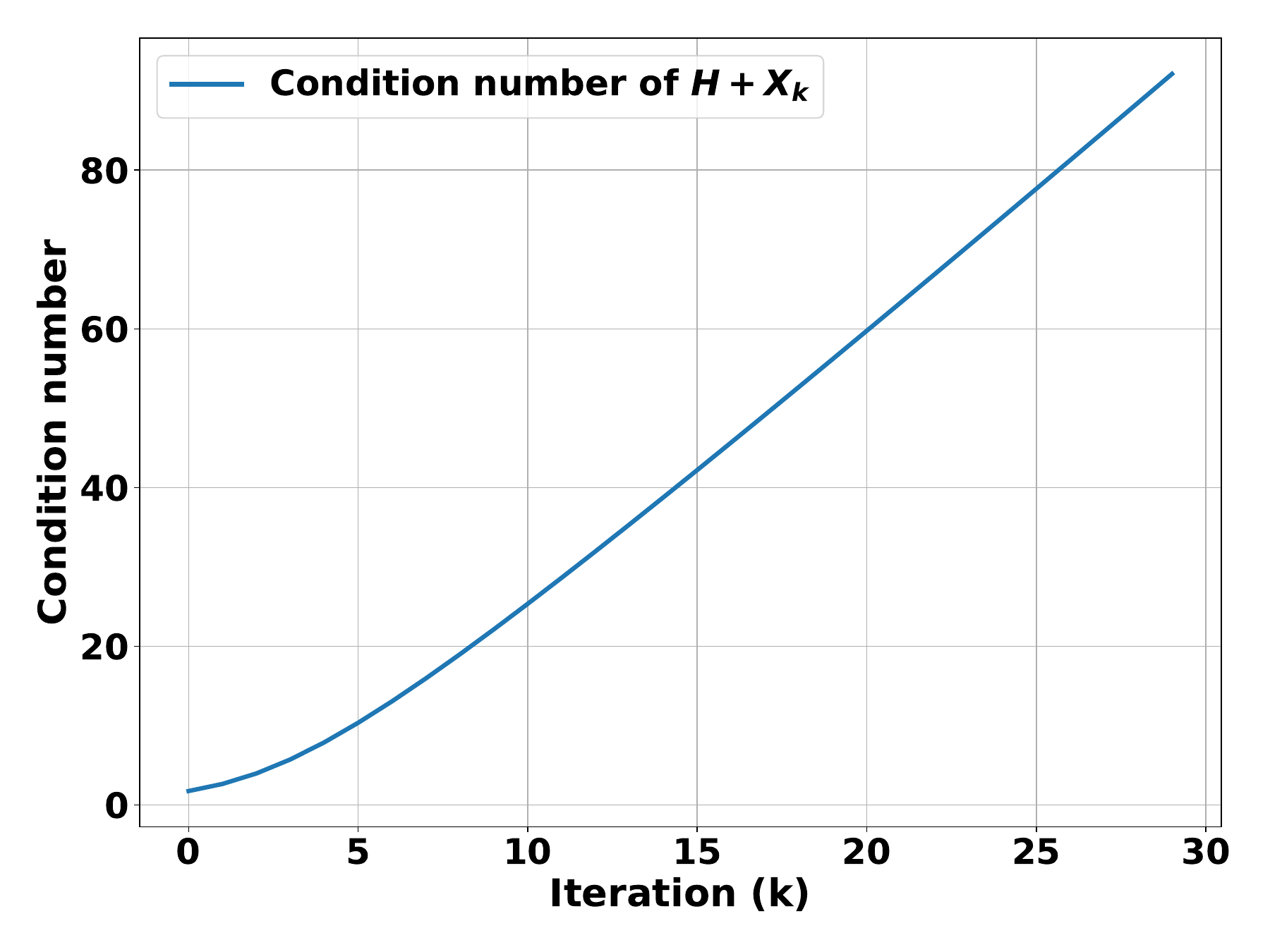}
  \end{minipage}
  \hfill
  \begin{minipage}[b]{0.45\textwidth}
    \centering
    \includegraphics[width=\textwidth]{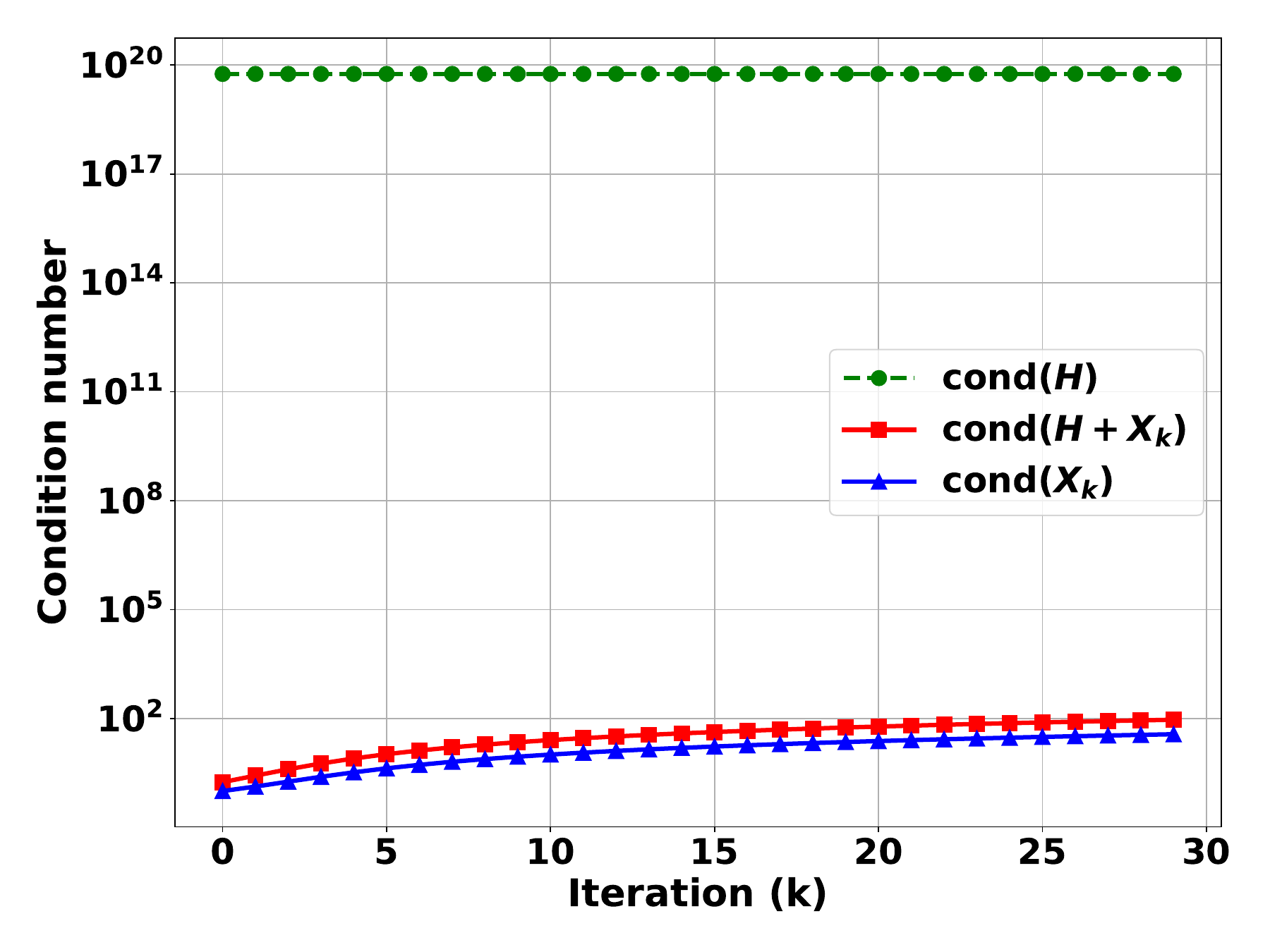}
  \end{minipage}
  \caption{We generate the Hilbert matrix $H \in \mathbb{R}^{200 \times 200}$. We plot the condition number of $H+X_k$ where $X_k$ is the $k$-th iterate of the fixed-point algorithm~\ref{eq:sqrt_fp} and compare this to the condition number of $H$. Clearly, $H+X_k$ is much better conditioned than $H$. This trend also holds for other very ill-conditioned matrices.}
  \label{fig:cond_num_sqrt}
\end{figure}

\subsection{Karcher Mean}
In this section, we compare the performance of the CCCP approach, i.e., the fixed point approach given by Eq.~\ref{eq:karcher_fp}, to RGD and RCG for the Karcher mean problem (Eq.~\ref{problem:sdiv_karcher_mean}).

\paragraph{Data generation} 
We focus on the medium-conditioned case, where we sample $G_{1}, \ldots, G_m $ random matrices, each with i.i.d standard Gaussian entries, and construct the data points $A_k \defas G_kG_k^\top$. A proxy for the true optimum is obtained by averaging the last iterate of the three algorithms upon convergence.

\paragraph{Results}
Figures~\ref{fig:karcher_mean_100_100} and~\ref{fig:karche_mean_100_500} show the convergence of all three algorithms.
We see that the CCCP algorithm exhibits superior runtime  compared to the two Riemannian first-order methods. Notably, the gap between CCCP and the two gradient-based approaches only widens as we increase the dimensionality and number of data samples.

\begin{figure}[htbp]
  \centering
  \begin{minipage}[b]{0.45\textwidth}
    \centering
    \includegraphics[width=\textwidth]{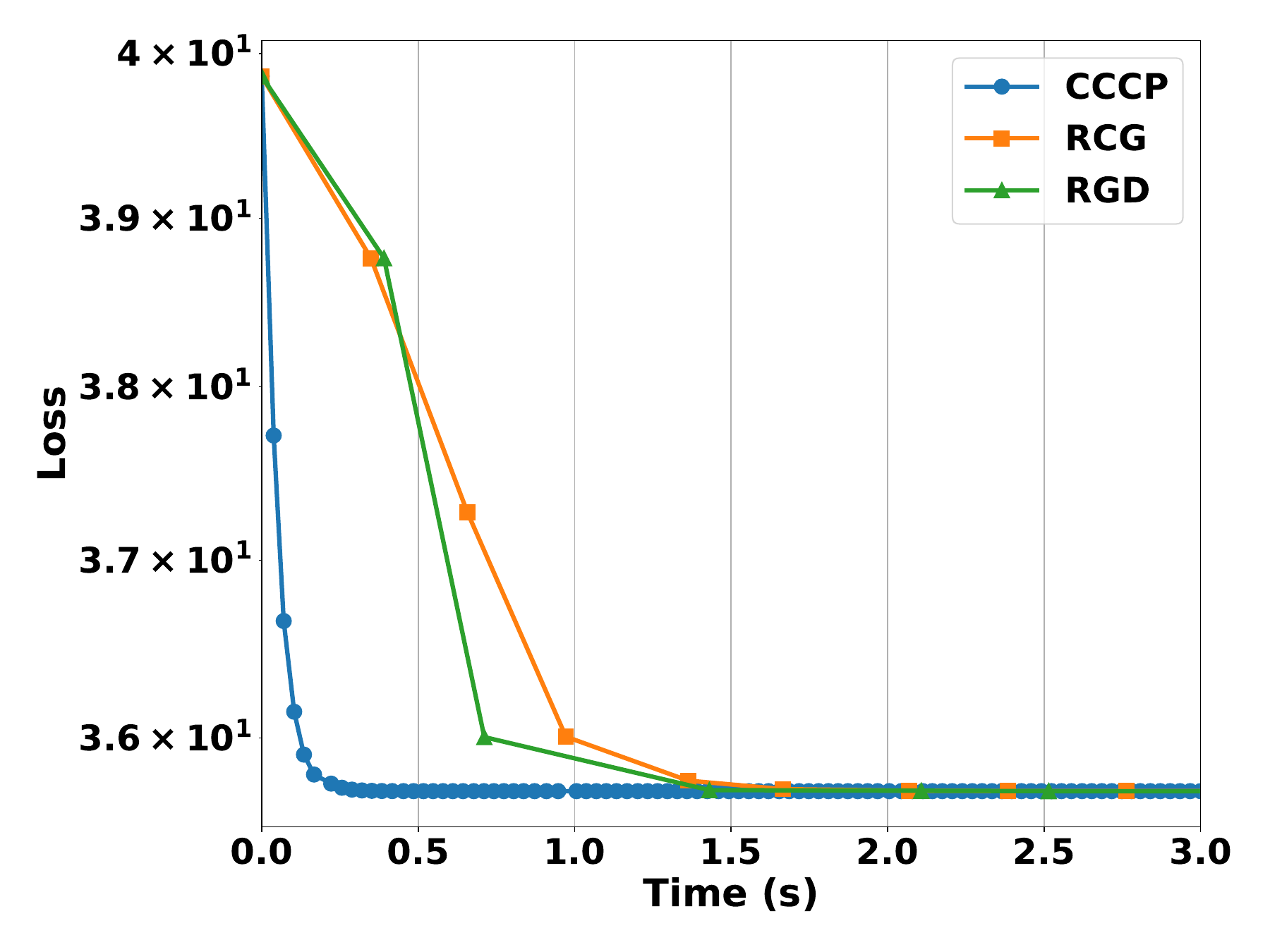}
  \end{minipage}
  \hfill
  \begin{minipage}[b]{0.45\textwidth}
    \centering
    \includegraphics[width=\textwidth]{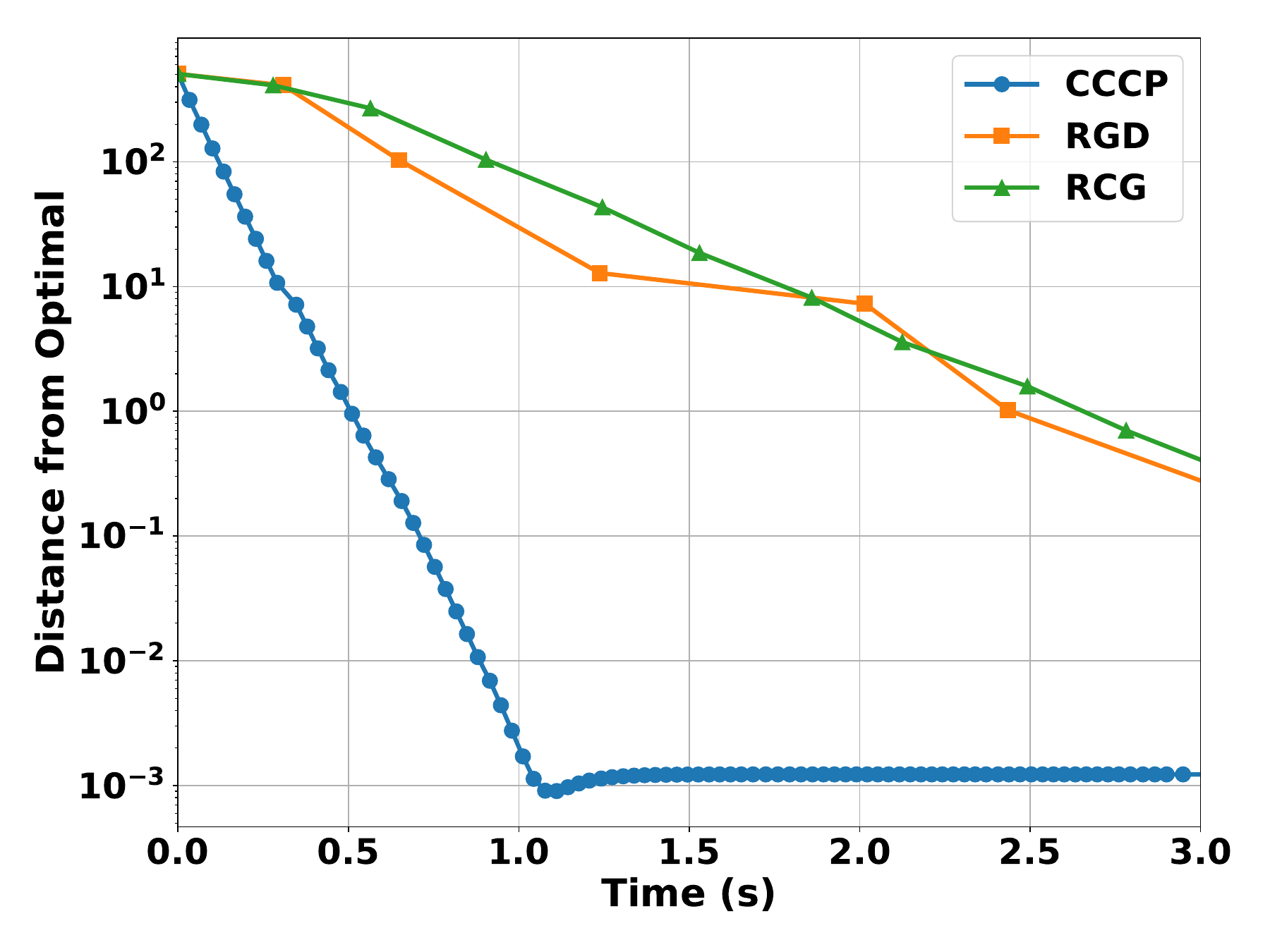}
  \end{minipage}
  \caption{\textbf{Karcher Mean.} $m=100$ and $d=100$. . CCCP demonstrates superior runtime complexity.}
  \label{fig:karcher_mean_100_100}
\end{figure}

\begin{figure}[htbp]
  \centering
  \begin{minipage}[b]{0.45\textwidth}
    \centering
    \includegraphics[width=\textwidth]{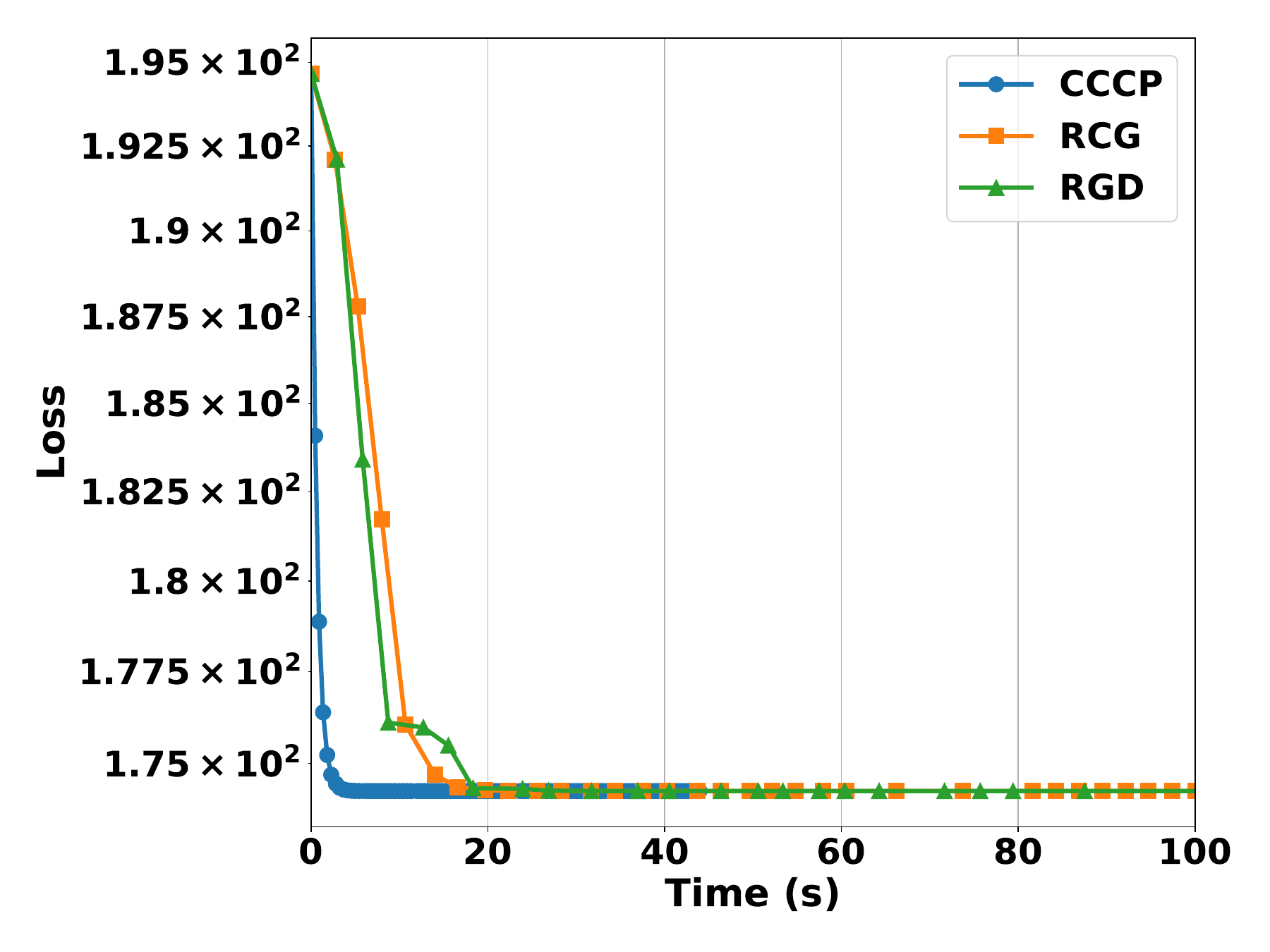}
  \end{minipage}
  \hfill
  \begin{minipage}[b]{0.45\textwidth}
    \centering
    \includegraphics[width=\textwidth]{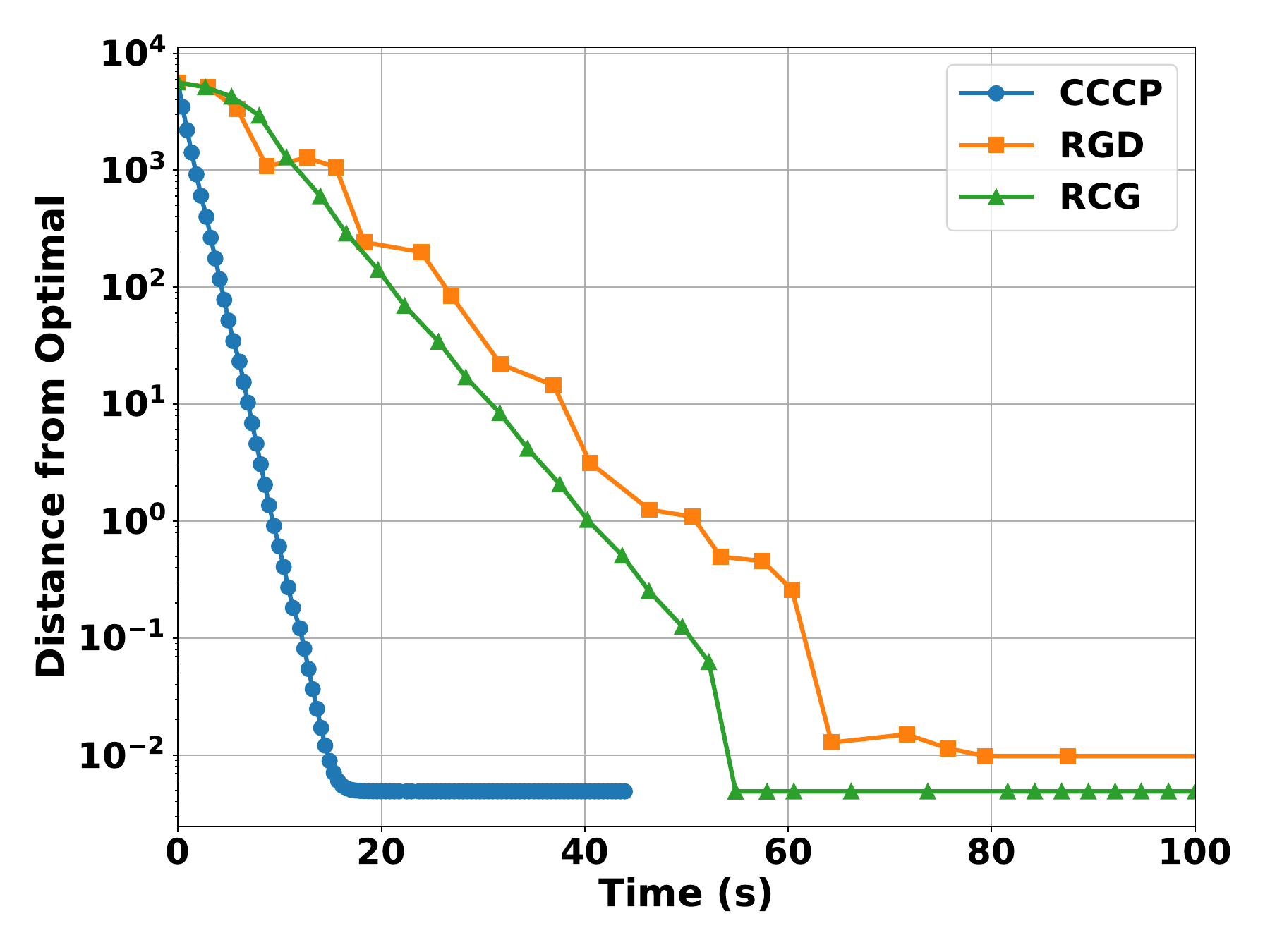}
  \end{minipage}
  \caption{\textbf{Karcher Mean.} $m=100$ and $d=500$. We observe that the gap between the runtime performance between CCCP and the benchmarks widens as we take the dimension $d$ to be larger.}
  \label{fig:karche_mean_100_500}
\end{figure}

\begin{figure}[htbp]
  \centering
  \begin{minipage}[b]{0.45\textwidth}
    \centering
    \includegraphics[width=\textwidth]{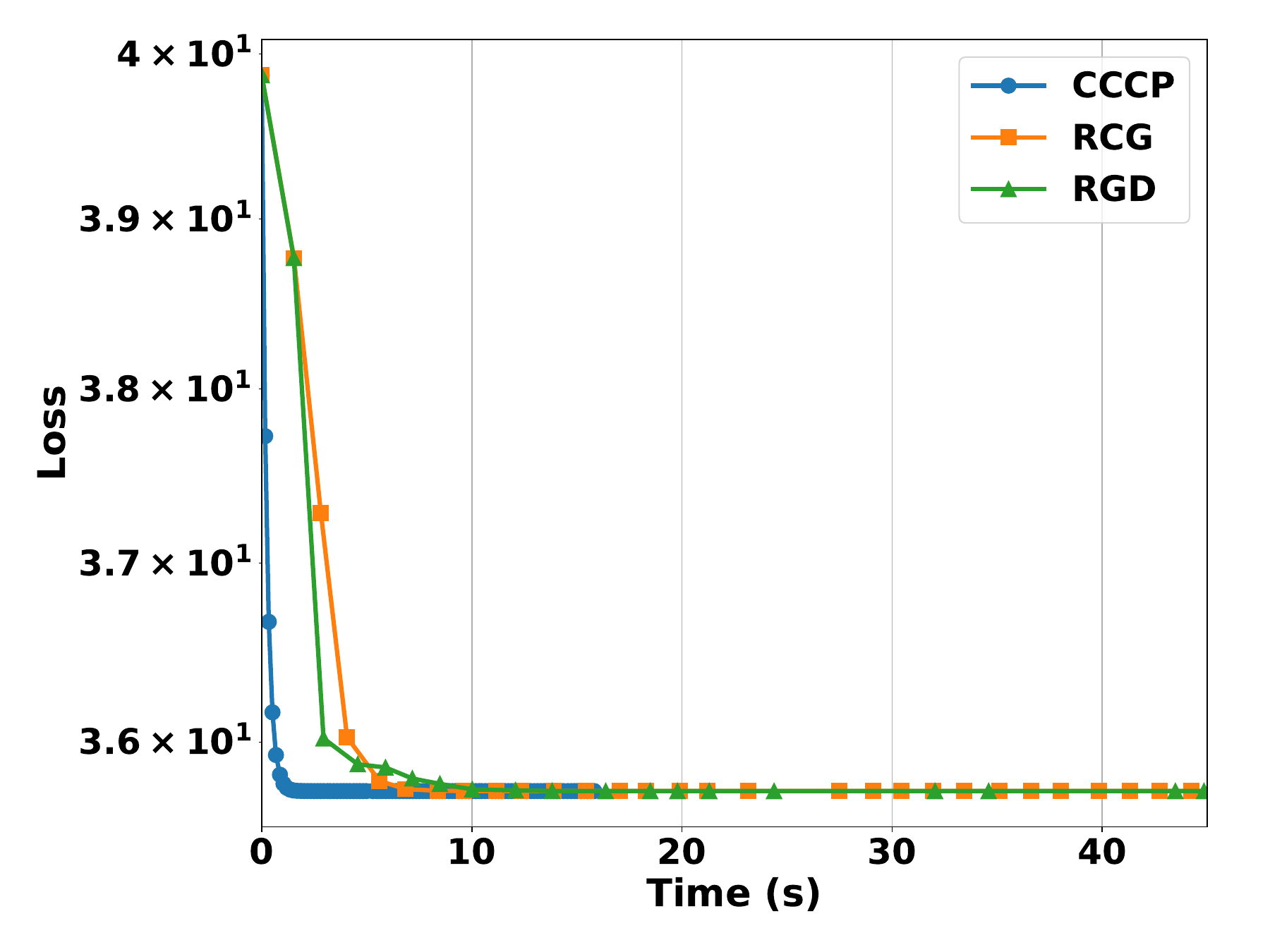}
  \end{minipage}
  \hfill
  \begin{minipage}[b]{0.45\textwidth}
    \centering
    \includegraphics[width=\textwidth]{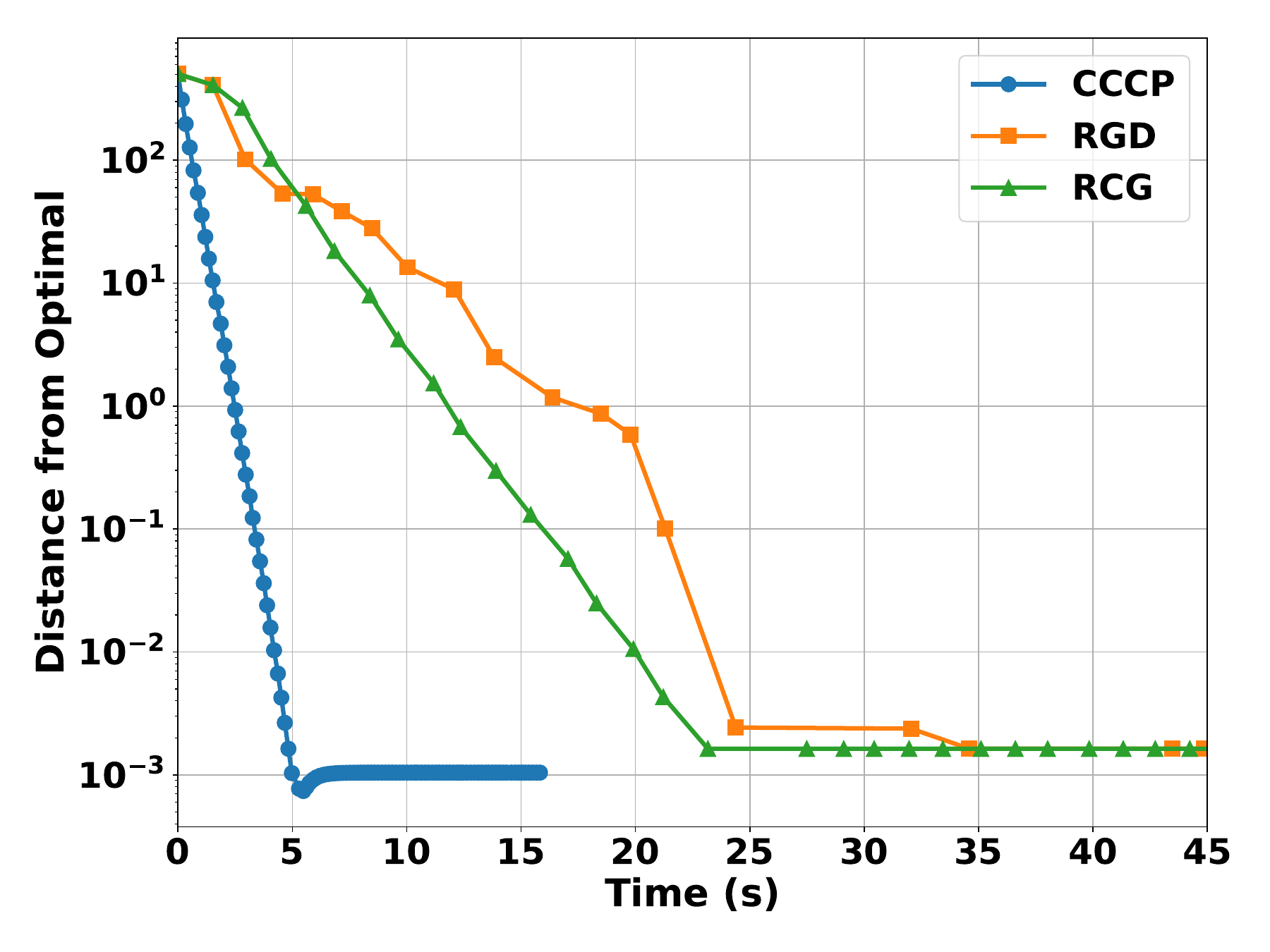}
  \end{minipage}
  \caption{\textbf{Karcher Mean.} $m=500$ and $d=100$. We observe that the gap between the runtime performance between CCCP and the benchmarks widens as we take the number of samples $m$ to be larger.}
  \label{fig:karcher_mean_500_100}
\end{figure}

\newpage
\subsection{Optimistic Gaussian Likelihood}\label{section:gaussian_aux_mle}
In this section, we test the performance of CCCP on the problem of computing the optimistic likelihood, introduced in sec.~\ref{sec:GaussianMLE_intro}:

\begin{equation}\label{exp:sdiv_ball_MLE_Problem}
\begin{aligned}
    &\argmin_{\Sigma \in \pd} \left\{ \hat{\phi}(\Sigma) \defas \tr\left(S \Sigma^{-1}\right) + \log \det \Sigma  + \beta \delta_S^2\left(\Sigma, \hat{\Sigma}\right) \right\}. 
\end{aligned}    
\end{equation}

\paragraph{Data Generation}
We follow an experimental setup similar to that of~\citep{Nguyen2019CalculatingOL}. In particular, we generate the true covariance $\Sigma$ and its estimate $\hat{\Sigma} \in \pd$ as follows. First we draw a Gaussian random matrix $A$ with i.i.d. entries $A_{ij} \sim \mathcal{N}(0,1)$. Then we symmetrize and ensure it is positive definite via $\Sigma = \frac{1}{2}\left(A + A^\top\right) + \delta I$. To construct $\hat{\Sigma}$ we conduct the eigenvalue decomposition $\Sigma = Q \Lambda Q^\top$ and replace the eigenvalues in $\Lambda$ with a random diagonal matrix $\hat{D}$ whose diagonal elements are sampled independently and uniformly from $\{1, 2, \ldots, 50\}.$

Our experiments in Figure~\ref{fig:Gaussian_aux_mle_30}
illustrates the output of Algorithm~\ref{alg:CCCP_on_MLE} and its distance from the true covariance $\Sigma$. The trend of the curves matches the intuition that increasing $\beta > 0$ and therefore increasing the confidence in $\hat{\Sigma}$ results in a solution closer to $\hat{\Sigma}$. We also note that higher regularization $\beta > 0$ results in faster convergence as exhibited in Figure~\ref{fig:Gaussian_aux_mle_30}. Moreover, Table~\ref {table:gaussian_aux_norm_master} illustrates the interpolation behaviour of $\hat{\Sigma}_\beta$ between $S$ and $\hat{\Sigma}$ as a function of $\beta$.

\begin{figure}
    \centering
    \includegraphics[width=0.43\linewidth]{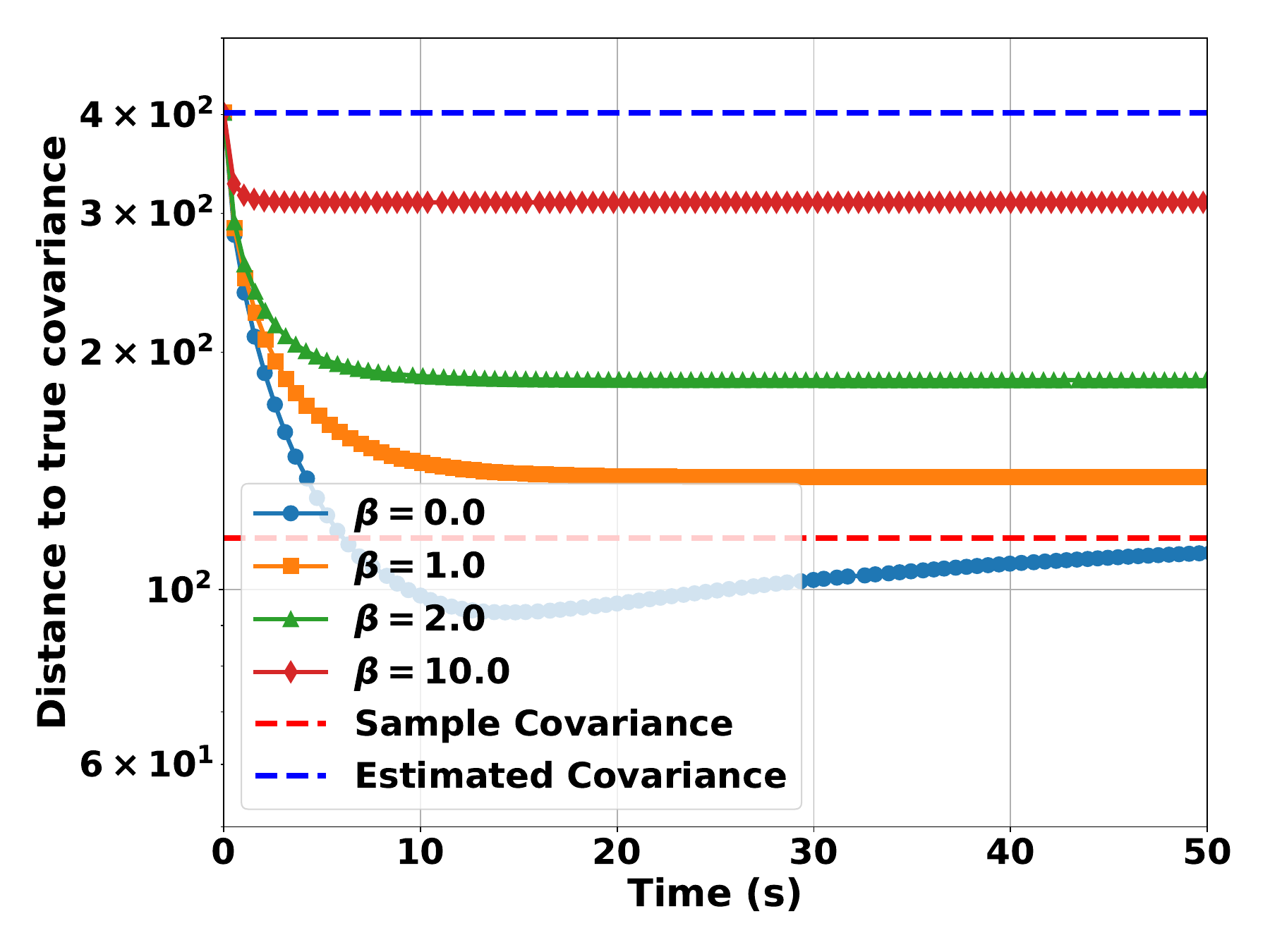}
    \includegraphics[width=0.43 \linewidth]{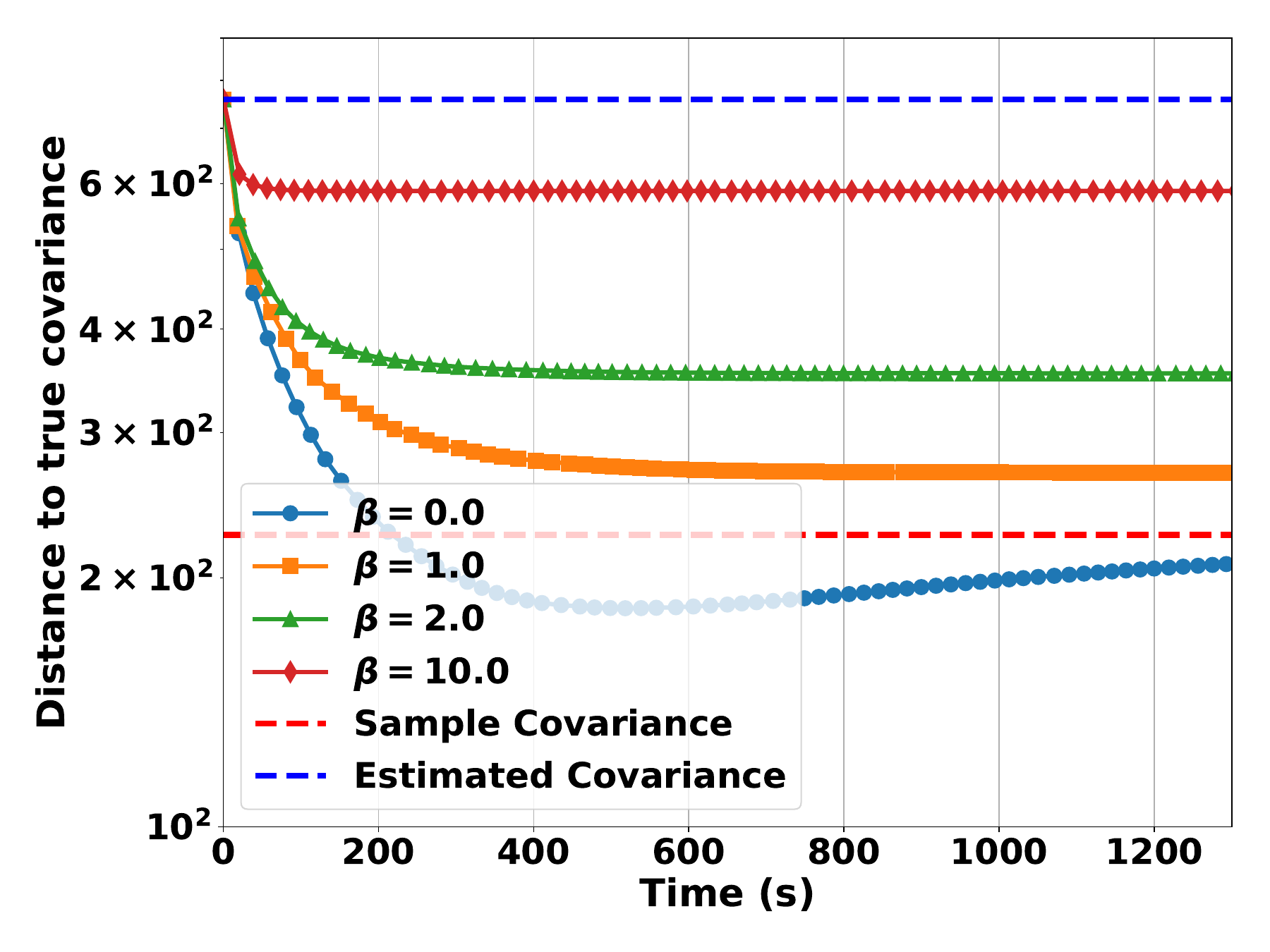}
    \caption{\textbf{Algorithm~\ref{alg:CCCP_on_MLE} on Gaussian Optimistic Likelihood.} We sampled $n=100$ independent Gaussian vectors of dimension $d=30$ for the left plot. Meanwhile, the right plot was generated with $n=1500$ and $d=100$. We initialized our iterate at our estimate $\hat{\Sigma}$.
    As we increase $\beta$, Algorithm~\ref{alg:CCCP_on_MLE} converges to a solution $\hat{\Sigma}_\beta$ closer to $\hat{\Sigma}$. At $\beta = 0$, the algorithm converges to the sample covariance, i.e., $\hat{\Sigma}_\beta = S$. Refer to table~\ref{table:gaussian_aux_norm_master} for the distances between $\|\hat{\Sigma}_\beta  - \Sigma^*\|_F$ and $\|\hat{\Sigma}_\beta - \hat{\Sigma}\|_F$ for varying $\beta$.}
    \label{fig:Gaussian_aux_mle_30}
\end{figure}

\begin{table}[htbp]
  \centering
  \begin{minipage}{0.45\textwidth}
    \centering
    \begin{tabular}{|c|c|c|}
    \hline
    $\beta$ & $\|\hat{\Sigma}_\beta - S\|_F$ & $\|\hat{\Sigma}_\beta - \hat{\Sigma}\|_F$ \\ \hline
    0  & 0.149   & 412.009 \\ \hline
    1  & 116.545 & 296.022 \\ \hline
    2  & 178.140 & 234.583 \\ \hline
    10 & 316.485 & 96.094  \\ \hline
    \end{tabular}
  \end{minipage}
  \begin{minipage}{0.45\textwidth}
    \centering
    \begin{tabular}{|c|c|c|}
    \hline
    $\beta$ & $\|\hat{\Sigma}_\beta - S\|_F$ & $\|\hat{\Sigma}_\beta - \hat{\Sigma}\|_F$ \\ \hline
    0  & 21.741  & 780.233 \\ \hline
    1  & 231.662 & 566.262 \\ \hline
    2  & 351.602 & 446.663 \\ \hline
    10 & 617.255 & 180.798 \\ \hline
    \end{tabular}
  \end{minipage}
  \caption{$\hat{\Sigma}_\beta$ denotes the output of Algorithm~\ref{alg:CCCP_on_MLE} with different $\beta$ values. The left table corresponds to the case $n=100$ and $d =30$. The right table corresponds to $n=1500$ and $d=100$.}
  \label{table:gaussian_aux_norm_master}
\end{table}

\subsection{Optimistic Multivariate T-Likelihood}
 We perform an experiment similar to that in the previous section for Algorithm~\ref{alg:aux_multivar_T}.
 We generate $d$-dimensional random vectors $x_1, \ldots, x_n \stackrel{\text{iid}}{\sim} \operatorname{MVT}(\Sigma; \nu)$ and observe the output of Algorithm~\ref{alg:aux_multivar_T} for different values of $\gamma$. We observe that incorporating $\hat{\Sigma}$ improves the distance from optimality. However, we observed that this method requires a high number of data points (See Figure~\ref{fig:multivar_t_aux_high_samples}).

\begin{figure}[ht]
    \centering
    \includegraphics[width=0.43\linewidth]{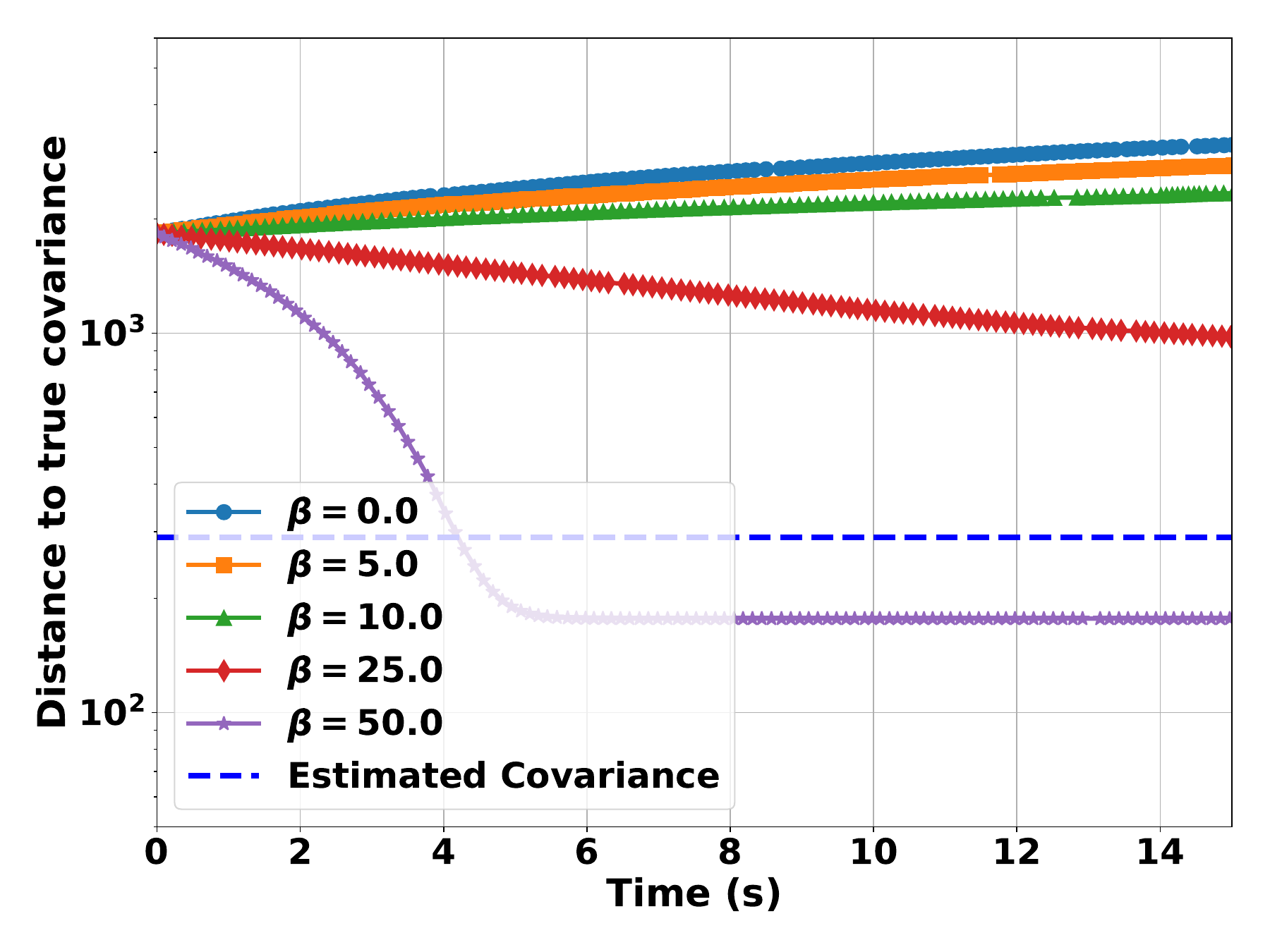}
    \caption{\textbf{Multivariate T-Distribution Optimistic Likelihood.} We sampled $n=1200$ i.i.d. multivariate t vectors of dimension $d=15$. The plot indicates similar behaviour of that to the Gaussian optimistic likelihood problem: higher regularization encourages solution to be nearer $\hat{\Sigma}$ and also exhibits faster convergence.}
    \label{fig:multivar_t_aux_figs}
\end{figure}

\begin{figure}[ht]
    \centering
    \includegraphics[width=0.43\linewidth]{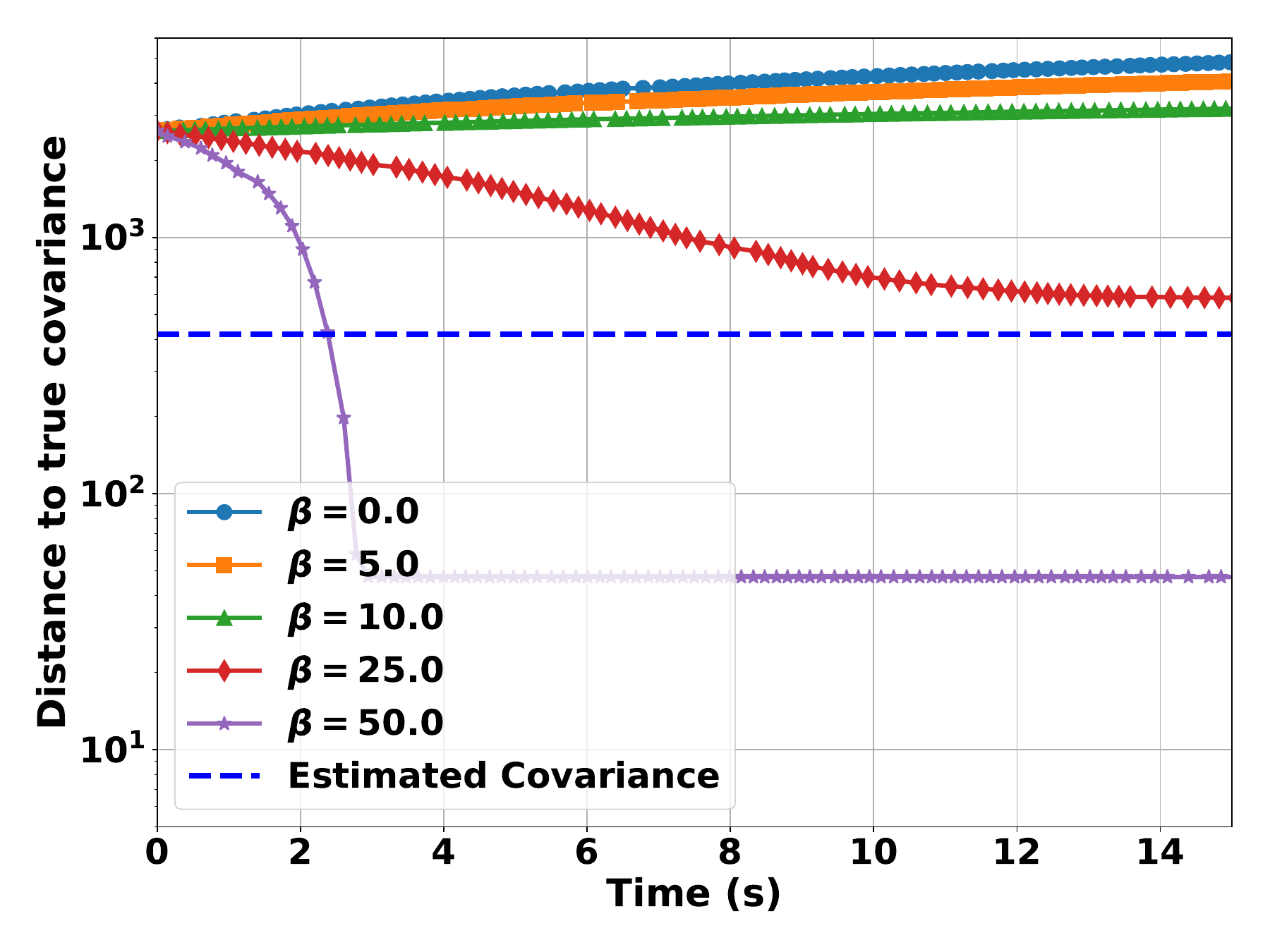}
    \includegraphics[width=0.43\linewidth]{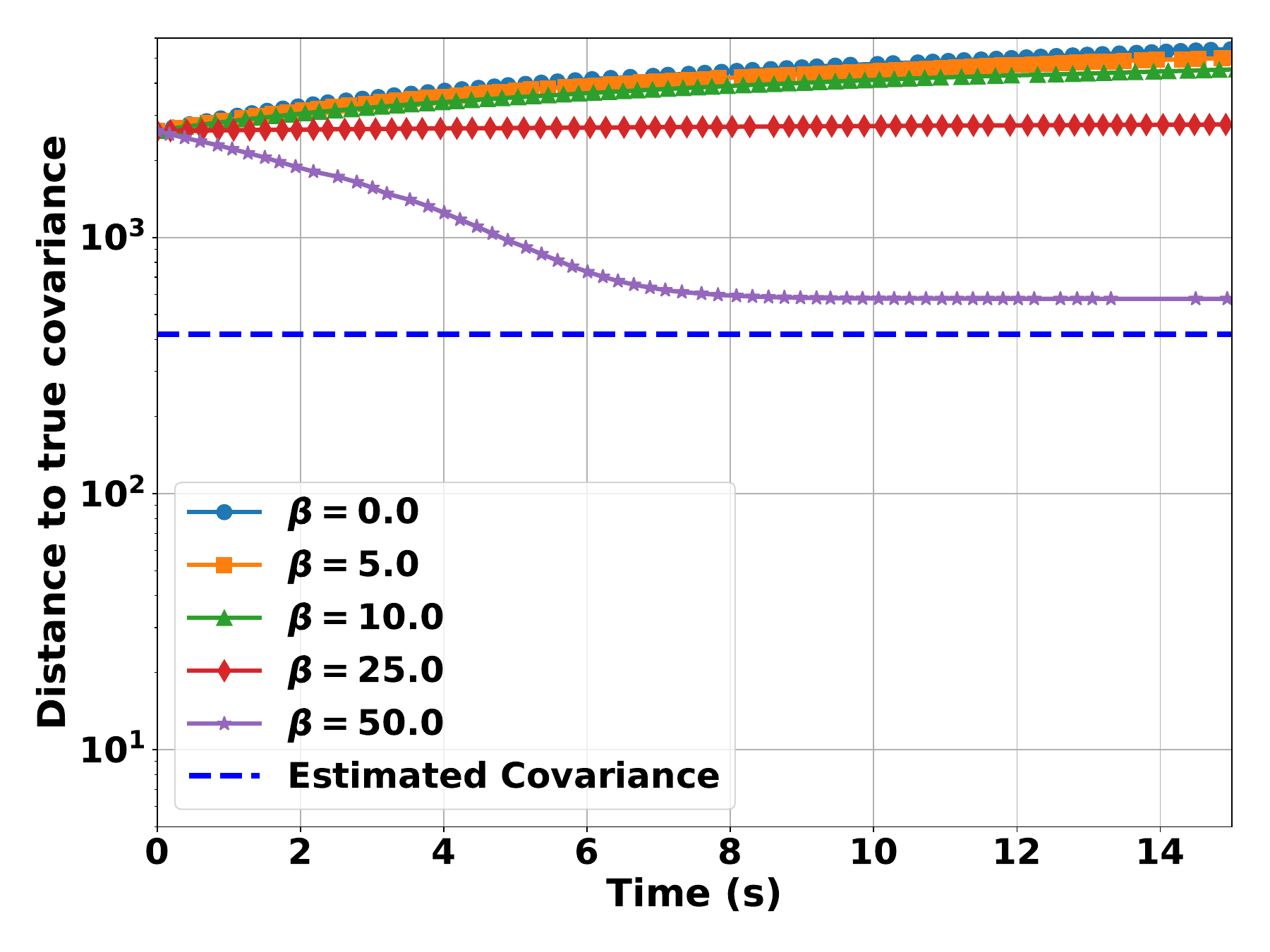}
    \caption{\textbf{Multivariate T-Distribution Optimistic Likelihood.} Both plots are $d=30$ and $\nu = 5$. The left plot is generated with $n=1500$ and the right with $n=3000$. This method requires a high number of samples for it to correspond with the theory. In particular, the right-side plot agreees with the theory: higher regularization corresponds to closer solutions to $\hat{\Sigma}$ in a monotonic fashion. The left-hand side with insufficient data points violates this.}
    \label{fig:multivar_t_aux_high_samples}
\end{figure}

\section{Discussion}
In this paper we introduced structured regularization approaches for constrained optimization on the SPD manifold. We considered different classes of constraints with a particular focus on sparsity and ball constraints. Our structured regularization approach relies on symmetric gauge functions, whose algebraic properties give rise to a modular framework that allows for designing regularizers that preserve desirable properties of the orginial objective, specifically geodesic convexity and difference of convex structure. We illustrate the utility of our approach on a range of data science and machine learning applications.

We believe that our proposed approach opens up new directions for constrained optimization on Riemannian manifolds that circumvents the potentially costly subroutines of standard constrained Riemannian optimization approaches, such as R-PGD and R-FW. While we focus on two specific classes of constraints for most of the paper, we believe that our approach could be applied much more broadly. Our discussion on disciplined programming with symmetric gauge functions may serve as a starting point for future work in this direction. The introduction of new regularizers for structured constraints could significantly widen the range of applications in machine learning and data science. Furthermore, while this paper only discusses constrained optimization on the SPD manifold, we believe that many of the ideas could be extended to other Cartan-Hadamard manifolds.

\section*{Acknowledgements}
We thank Bobak Kiani and Vaibhav Dixit for helpful discussions and comments.\\ 

\noindent This work was supported by the Harvard Dean’s Competitive Fund for Promising Scholarship and NSF award 2112085. AC is partially supported by a NSERC Postgraduate Fellowship.

\newpage
\bibliographystyle{plainnat}
\bibliography{refs}


\end{document}